\documentclass[notitlepage,reqno,11pt]{amsart}
\usepackage[margin=1in, marginparwidth=0.9in]{geometry} 

\RequirePackage[english]{babel}
\RequirePackage[shortlabels]{enumitem}
\RequirePackage{amsmath,amsfonts,amssymb,amsthm,amsbsy,amscd}
\RequirePackage[toc,page]{appendix}
\RequirePackage{color}
\RequirePackage{fontenc,xfrac,placeins,xcolor,booktabs,bbm,bm}
\RequirePackage{graphics,epsfig,graphicx}
\RequirePackage{soul}
\RequirePackage[colorlinks,allcolors=blue]{hyperref}
\RequirePackage[capitalise]{cleveref}
\RequirePackage[margin=1in]{geometry}

\usepackage{comment}
\usepackage{parnotes}
\usepackage{cancel}
\numberwithin{equation}{section}

\newcommand{\hcm}[1]{\hspace{#1 cm}}  

\newtheorem{theo}{Theorem}[section]
\newtheorem{prop}[theo]{Proposition}

\newtheorem{lem}[theo]{Lemma}

\newtheorem{rem}[theo]{Remark}
\newtheorem{hyp}{Assumption}

\newcommand{\bD}{\mathbb{D}}
\newcommand{\bE}{\mathbb{E}}
\newcommand{\PR}{\mathbb{P}}
\newcommand{\RR}{\mathbb{R}}
\newcommand{\bX}{\mathbf{\mathbb{X}}}
\newcommand{\bY}{\mathbf{\mathbb{Y}}}

\newcommand{\cA}{\mathcal{A}}
\newcommand{\cB}{\mathcal{B}}
\newcommand{\cC}{\mathcal{C}}
\newcommand{\sD}{\mathcal{D}}
\newcommand{\cE}{\mathcal{E}}
\newcommand{\cF}{\mathcal{F}}
\newcommand{\cG}{\mathcal{G}}
\newcommand{\cH}{\mathcal{H}}
\newcommand{\cI}{\mathcal{I}}
\newcommand{\cL}{\mathcal{L}}
\newcommand{\cM}{\mathcal{M}}
\newcommand{\cR}{\mathcal{R}}
\newcommand{\cS}{\mathcal{S}}
\newcommand{\cX}{\mathcal{X}}
\newcommand{\cY}{\mathcal{Y}}

\newcommand{\mfA}{ \mathfrak{A}}
\newcommand{\mfD}{ \overline{\Delta}^N_D} 
\newcommand{\mfE}{ \mathfrak{E}}
\newcommand{\mfF}{ \mathfrak{F}}
\newcommand{\mfL}{ \mathfrak{L}}
\newcommand{\mfM}{ \mathfrak{M}}
\newcommand{\mfU}{ \mathfrak{U}}
\newcommand{\mfV}{ \mathfrak{V}}

\newcommand{\eps}{\varepsilon}
\newcommand{\Rd}{\mathrm{d}}
\newcommand{\Var}{\text{\rm Var}}

\newcommand{\II}[1]{[\![#1]\!]}
\newcommand{\Ninf}[1]{\| #1 \|_\infty}
\newcommand{\idc}[1]{\mathbbm{1}_{\left \lbrace #1 \right \rbrace}}
\newcommand{\wtd}[1]{\widetilde{#1}}
\newcommand{\wht}[1]{\widehat{#1}}

\newcommand{\ttl}{\large Stochastic heterogeneous SIR model with infection-age     dependent infectivity 
 on large random graphs}

\begin{document}

\title[]{\ttl}

\author[Guodong \ Pang]{Guodong Pang}
\address{Department of Computational Applied Mathematics and Operations Research,
George R. Brown School of Engineering and Computing, 
Rice University,
Houston, TX 77005}
\email{gdpang@rice.edu}

\author[{\'E}tienne \ Pardoux]{{\'E}tienne Pardoux}
\address{Aix--Marseille Universit{\'e}, CNRS, Centrale Marseille, I2M, UMR \ 7373 \ 13453\ Marseille, France}
\email{etienne.pardoux@univ-amu.fr}

\author[Aur\'{e}lien \ Velleret]{Aur\'{e}lien Velleret} 
\address{Université Paris-Saclay, Université d’Evry Val d’Essonne, CNRS, LaMME, \ UMR 8071
	\ 91037 Evry, France} 
\email{aurelien.velleret@nsup.org}

\date{\today}

\begin{abstract} 
We study an individual-based stochastic SIR epidemic model with infection-age dependent infectivity on a large random graph, capturing individual heterogeneity and non-homogeneous connectivity.
Each individual is associated with particular characteristics (for example, spatial location and age structure), which may not be i.i.d., and is represented by a particular node.
The connectivities among the individuals are given by a non-homogeneous random graph, whose connecting probabilities may depend on the individual characteristics of the edge.  To each
 individual is associated  a random  infectivity function  of its infection age, which is allowed to depend upon the individual characteristics. 
We use measure-valued processes to describe the epidemic evolution dynamics, tracking the infection age of all individuals, and their associated characteristics.  
We consider the epidemic dynamics as the population size grows to infinity 
under a specific scaling of the connectivity graph
related to the convergence to a graphon. 
In the limit, we obtain a system of measure-valued equations, which can be also represented as a 
PDE model on graphon, and reflects the heterogeneities in individual
characteristics and social connectivity. 
\end{abstract}


\maketitle

\allowdisplaybreaks

\section{Introduction}

In the present paper, we are interested in the large population limit of the stochastic  model of an epidemic that spreads
among an heterogeneous population with random connectivities. Such models have been studied to an extent in the literature.  For example, in \cite{BBT2020}, an epidemic model with an age structure and various social activity levels is studied to understand the effect of population heterogeneity on herd immunity. 
In addition to `age' as an obvious heterogeneous factor, spatial locations and other characteristics/features may also  contribute to the individual heterogeneity. 
Another source of heterogeneity arises from the individual connectivities (such as households, workplaces, communities and social activities), which is often modeled as a non-homogeneous random graph. 
For example, epidemic models on 
random graphs with given degrees,
typically under the configuration model 
are studied in \cite{andersson1998,volz2008sir,shang2013mixed,BP2012,BS2013,janson2017,ferreyra2021,lashari2021,luczak2024note,clancy2025},
on weighted  (configuration) graphs \cite{BDL2011,deijfen2011,BL2012,SB2019}, on dynamic (evolving) graphs 
\cite{BBLS2019,BB2022,jiang2019,DY2022,CHY25,FMO25,HR2024}. 
(See also the relevant models  on random networks with household structures in \cite{BST2009,BS2012,BS2018,kubasch2023}, and
on multilayer networks in \cite{jacobsen2018,kubasch2023}.)

Much relevant to our work are the large population limits for epidemic models on random graphs.  
In particular, \cite{DDMT2012,JLW2014} proved a functional law of large numbers (FLLN) for a Markovian SIR model on a configuration model graph with specified degree distributions and edges being randomly matched, and established the measure-valued limit as a system of nonlinear differential equations, which verifies the conjecture in \cite{volz2008sir}. 
In \cite{khudabukhsh2022}, a functional central limit theorem (FCLT) is established for a similar Markovian SI model for the total count processes. 
In \cite{BHI2024}, an FLLN is established for a Markovian SIR model on a stochastic block model. 
In \cite{DDMT2012,JLW2014, khudabukhsh2022, BHI2024},
the degree distribution converges to a finite limit as the number of nodes tends to infinity
and the number of edges is thus of the same order as the number $n$ of nodes.
Such a description with very sparse graphs 
leads to limits that are of distinct nature
from the one of dynamics on a graphon. 
On the other hand, for dense connectivities among individuals, one can derive large population approximations of the epidemic dynamics as PDEs on graphons. 
In \cite{KHT2022}, an FLLN is established for density-dependent Markov processes with finite state space on large random graphs, which includes the Markovian SIS model 
with individual heterogeneity on graphs as a special case, and a PDE on graphon is derived as the limit. 
In \cite{DFG+24}, an FLLN is established for a Markovian SIS model with a general form of individual heterogeneity on random graphs, where a PDE on graphon for the measure-valued state descriptors is also derived. 
In the latter two  references (\cite{KHT2022, DFG+24})
	the scalings of the parameters allow for the random generation of graphs 
whose number of edges is of order $n^{1+a}$,
where $n$ is the number of nodes and 
$a$ can be freely taken in $(0, 1]$.
This extends the classical assumption of a dense graph for the convergence to a graphon, where the number of edges is of order $n^2$. 
However, all these FLLN results are about  epidemic models on random graphs
that are Markovian, where the infectious durations or recovery times are exponentially distributed. 
In the present paper, we establish an FLLN for a non-Markovian SIR model on large random graphs that results in a measure-valued limit on a graphon. 

Non-Markovian stochastic epidemic models with a homogeneous population (no age structure, homogeneous social connectivity, etc.) have been recently studied, see the survey \cite{FPP2022}. Although these models offer much more possibilities to fit observational data on infection profiles,
many of the classical tools in probability 
cannot  be directly exploited. 
In particular, the standard epidemic models with a constant infection rate and a general infection duration distribution are recently studied in \cite{PP-2020} (see also \cite{wang1977gaussian,wang1975limit} for Gaussian approximations and \cite{reinert1995} for measure-valued state descriptors and PDE limit for the FLLN), and epidemic models with
 varying infectivity are studied in \cite{FPP2020b,PP2020-FCLT-VI,FPP2022-MPMG}, where deterministic or stochastic integral equations are derived for the FLLNs and FCLTs, respectively. 
When adding the age of infection into the variables of the model,
  measure-valued processes have been used to describe the epidemic evolution dynamics in models with such infection-age dependent infectivity,  and the corresponding PDE and SPDE limits are established for the FLLNs and FCLTs in \cite{PP2021-PDE} and \cite{PP-2024-SPDE}, respectively. 
  This is extended to spatially dense SIR models in \cite{pp-2026}, where the FLLN is established with a PDE limit on a graphon. 
  An epidemic model with contact tracing and general infection duration distributions is studied in \cite{duchamps2023}, where measure-valued processes are used to describe the dynamics and the FLLN is established
with a PDE limit. In \cite{kubasch2023}, an individual-based multi-layer SIR model with households and workplaces to take into account the social connectivity heterogeneity is studied, where measure-valued processes tracking remaining infectious times are used to describe the epidemic dynamics and a PDE limit is established for the FLLN. 
\smallskip

In this paper we consider an SIR model with infection-age dependent infectivity on a large random graph that captures
both individual and connectivity heterogeneities.
In particular,
connectivities among individuals are given by
a non-homogeneous random graph.
Each node on this graph represents
a single individual,
each associated with an individual characteristic/feature.
These may account for spatial locations,
age and/or social behaviors.
So they may be distributed on a compact subset of $\RR^d$,
	as in the spatial SIR model considered in 
	\cite{pp-2026,KP2024spatial},  which we generalize in the simplest setting of pairwise interactions.
Concerning age structure or social activity levels, our model generalizes \cite{BBT2020} in that they only consider a finite number of compartments while our model includes a possibly continuous age model. 
The random connectivity of each edge is then assumed to depend on the individuals' characteristics in a general manner.  The connectivities indicate various levels of social activities, which may for instance depend on the spatial locations or age.
The random graphs that we consider are of the same form as the ones in \cite{DFG+24}. 
The main difference is that the connectivity of each edge 
was assumed in \cite{DFG+24}
to be a deterministic function of the pair of individuals' characteristics,
while we allow for an additional degree of randomness. In \cite{KHT2022}, there is no heterogeneity in the contact rate and the dependency of the graph on individual types is more restrictive. 
In \cite{BHI2024}, the study of SIR dynamics on large graphs only focuses on stochastic block models.

Each individual is moreover associated with a random varying infectivity function, which reflects the force of infection of each individual at any time after infection, and depends as well on the individual characteristics. 
As a result, the infectious duration is also individualized (we express the corresponding distribution 
as a function $F_x(\cdot)$ of the individual characteristic $x$). 

Individuals are grouped into three compartments: Susceptible, Infected and Recovered. Infections are generated by the interactions between  susceptible and infected individuals. 
Each individual is thus exposed
	to a specific force of infection
	which changes over time 
and  is given by aggregating
the weighted infectivity levels
of the individuals that are connected to him/her in the graph.
The whole dynamic of the epidemic
depends on the non-homogeneous random graph and on the random infectivity functions 
through the expression of the force of infection that is exerted from its infectious neighbors upon each individual.

To describe the evolution dynamics, we use three measure-valued processes for the three compartments, where  at each time $t$ the measure for the infected process is over both the individual characteristics and the infection ages, while the measure for the susceptible and recovered processes is only over the individual characteristics. 
We prove a FLLN for these measure-valued processes 
when the population size goes to infinity.  
As the population size increases, so does the connectivity graph. 
We impose conditions on the connectivity probabilities to keep the graph consistency as it grows,
which notably covers the case of a graphon in the limit. 
The limits for the FLLN are given by a set of measure-valued equations, from which we further 
characterize the measure-valued process tracking the infection ages of infected individuals
as the unique solution of a PDE. 
To note, the dynamics described from the set of measure-valued equations
and the PDE
may be seen as evolving on a graphon.

The PDE model is linear, with a boundary condition  given by the product of the measure-valued susceptible process and the aggregate force of infection.
A specific form of Picard iteration for these quantities
is used to deduce the existence of solutions. 
Solution uniqueness follows from classical methods based on the characteristics of the equation,
as in \cite[Section~3.2]{Ev10}, e.g.,
adapted for measure-valued solutions.

\smallskip

The novelty of our paper is in the combination of two aspects:
heterogeneity of the individuals and of the graph of interactions on the one hand, 
and infection-age dependent infectivity on the other hand.
The second aspect has been considered recently 
by some of the authors of the present paper in a homogeneous situation.
We borrow from these works,
see notably \cite{FPP2022-MPMG},
the effective techniques consisting in the construction
of an intermediate model
for which the infection rates are derived from the deterministic limiting model.
Whereas the limiting infectivity function was previously described 
as acting on total count processes,
it is acting in the present paper
at the individual level, 
with a dependence on the characteristics of each individual 
through the limiting connectivity kernel 
and infectivity function.

The variance-based approach to approximating random sums by their expectations, previously used in \cite{FPP2022-MPMG} for treating the heterogeneity in the choice of the infectivity function, 
is effective in our framework 
with additional heterogeneity, see Lemma~\ref{lem_VN_bound}.
The conditions that we assume on the average infection rate per edge
and on the average variance in this infection rate
seem very close to be optimal. 

Nonetheless, the combination of both heterogeneities 
induces some complications
for the control of the error induced by the replacement of the original model 
by that simpler one, see in particular the proof of Lemma~\ref{lem_AN_bound}.

As in \cite{DFG+24}, 
	we have relaxed as much as possible
the regularity assumptions 
of the interaction kernel in the limit,
which we only assume to be almost everywhere continuous. 
In addition, we allow for the individual characteristics to take values in a state space as general as possible,
which implies that the state descriptors evolve as measure-valued processes.
Whereas the initial condition is generally assumed to be given through an independent and identically distributed (i.i.d.) sample,
we have merely assumed weak convergence in probability in terms of distributions.
The resulting difficulty concerns particularly 
the discrepancies in the initial conditions, 
see the proof of  Lemma~\ref{lem_EN_bound}, which is 
based on a new convergence result of independent interest
for weakly continuous kernels on general Polish spaces, see Proposition~\ref{prop_cvg_mu0}.

\subsection{Organization of the paper}
The rest of the paper is organized as follows. We give a summary of some common notations used throughout the paper in the next subsection. We provide a detailed model description and our assumptions in Section \ref{sec:model}. 
In Section \ref{sec:FLLN}, we state the main result of the paper. 
We establish the uniqueness of the solution to the limiting deterministic PDE model and establish some properties of that solution   in Section \ref{sec:limit}. The proofs of the convergence in the FLLN are given in Section~\ref{sec:proofs}, with additional technical supporting results proved
in Appendix~\ref{sec:appendix}. Concluding remarks are provided in Section~\ref{sec_disc_MT}. 

\subsection{Notation}
	We denote by $\mathbb{N} = \{1,2,\ldots\}$ the set of positive integers. 
	We also set $\II{1, n}=\{k\in \mathbb{N} \, \colon\, 1\leq k\leq n\}$
	for $n\in \mathbb{N}$. 
	For $a,b\in \RR$, we write 
	$a \wedge b$  for  the minimum between $a$ and $b$ and 
	$a \vee b$  for  the maximum between $a$ and $b$. 
	For a real-valued function defined on a set $\cX$, 
	we denote by $\Ninf{f}$ its
	supremum norm with 
	\[
	\Ninf{f}=\sup_{x\in\cX}   |f(x)|.
	\]
	We denote by $C_b(\cX)$ the set of continuous and bounded functions on a metric space $\cX$.

In what follows  $(\bX, \Rd)$ will denote a  Polish metric space, i.e., complete and separable. 
	We  denote   by
	$\cM_1(\bX)$  and $\cM(\bX)$ the  sets   of   probability 
	and 
	 finite measures   on~$\bX$, respectively. 
	For $\mu\in \cM(\bX)$ and a real-valued measurable function $f$ defined
	on $\bX$, we  will sometimes denote the integral of  $f$ with respect
	to      the      measure       $\mu$,      if      well-defined,      by
	$\langle \mu,\, f\rangle = \int_\bX f(x )\, \mu(\Rd x)= \int f\, \Rd \mu$.
We 	endow $\cM_1(\bX)$ and $\cM(\bX)$ with  the topology of weak convergence. 
$\cM_1(\bX)$ and $\cM(\bX)$ are Polish metric spaces with the Lévy-Prokhorov metric,
	cf. \cite[Section II.4]{perkins}.
We denote 
by $\bD(\RR_+,  \cM(\bX))$ the space of  
 paths from $\RR_+$  to $\cM(\bX)$	
that are right-continuous
with left limits (c{\'a}d-l{\'a}g).  This  space is
endowed  with the  Skorokhod   topology  (see, \textit{e.g.}, \cite[Chapter
3]{billingsley99}).
We shall also use the two following Skorokhod spaces:
$\bD(\RR_+,  \RR)$ and $\bD(\RR_+,  \cM(\bX\times \RR_+))$.

\section{Model description} \label{sec:model}

We consider a population consisting of  $N$ individuals, who are in one 
of the three states -- susceptible, infected or recovered at each time,  and may interact with each other according to a random graph described below. 
Susceptible individuals can be infected randomly through interactions with the infected ones. Once infected, an individual will  become infectious for a random period of time until recovery, and once recovered, he or she will no longer infect any susceptibles, nor become infected again.  
Let $ \cS^N(t),  \cI^N(t),  \cR^N(t)$ be the subsets of $\II{1, N}$ that denote the sets of susceptible, infected or recovered individuals at time $t$, respectively. The corresponding processes $S^N(t) = |\cS^N(t)|$,  $I^N(t) = |\cI^N(t)|$ and  $R^N(t) = |\cR^N(t)|$ denote the numbers of susceptible, infected or recovered individuals at  time $t$. 

In the limit of a large population,
	our goal is to relate the spread of the disease in the population of $N$ individuals
	to a dynamics acting on a Polish space $\bX$ 
	that accounts for the heterogeneity of individuals.
Any individual $i\in \II{1, N}$ is actually characterized by some random variable $X_i^N\in \bX$
 (which is fixed over time).
The characteristic $X_i^N$ of individual~$i$ 
is allowed to affect
the infectious contact rates with the other individuals
as well as the  evolution of its own infectivity level. 

\subsection{Infection-age dependent infectivity}\label{sec_def}
For an individual $i\in \cS^N(0)$, let  $\tau_i^N>0$ be the time at which he/she starts being infected, and let $A_i^N(t) = t - \tau_i^N$ be the infection age of the individual $i$ at time $t$
(by default it is equal to zero for $t<\tau_i^N$).
For an initially infected individual $j \in \cI^N(0)$, let $\tau_j^N = - A_j^N(0)$ denote the infection time before time $0$, so that $A_j^N(t) = t + A_j^N(0)$ is the infection age at time $t$. 

Let $\cF_0^N$ be the $\sigma$-field generated by the $(X_i^N, i\le N)$,
 $\cS^N(0)$, $\cI^N(0)$ and $(A_j^N(0), j\in \cI^N(0))$.
We shall generally consider the dynamics conditional on this sigma-field $\cF_0^N$.
Thus, we shall use the notations $\PR_0^N$ and $\bE_0^N$
	to express probabilities and expectations 
	conditional on $\cF_0^N$, respectively.

At time $t$, each individual $i \in \cI^N(t)$ who is infected has 
 an infectivity function depending randomly on  its  infection age:  
\begin{equation}\label{eq_def_lambdaiN}
	\lambda^N_i(A_i^N(t)), \quad t\ge 0\, , 
\end{equation}
where $\lambda^N_i$ is assumed not to  be identically  0. 
 In particular, for the initially infected individuals $j\in \cI^N(0)$,
 their infectivity level  at a given time $t>0$  is:
	\begin{equation*}
\lambda^{N, 0}_j(t) := \lambda^N_j(t + A_j^N(0))\,.
	\end{equation*}

\begin{rem}\label{rem_wht_lam}
A typical way to define these random functions $\lambda^N_i(\cdot)$ is specified as follows,
although we do not rely in our proofs on this representation.  
Let $\cY^N=\{Y_i^N\}_{i\le N}$ be 
an i.i.d. sequence of uniform random variables on $[0, 1]$, independent from $\cF_0^N$.
	Let $\wht \lambda$ be a deterministic measurable function from $\bX\times [0, 1]\times\RR_+$
to $\RR_+$ that is càd-làg in its last component.
Then we can define the random functions $\lambda^N_i(\cdot)$ for each $i$ as 
\begin{equation}\label{eq_def_lamY}
	\lambda^N_i(a) = \wht \lambda(X_i^N, Y_i^N, a)\,. 
\end{equation}
The variables $\{Y_i^N\}$ capture the random factors 
associated with the infectivity of the individuals
apart from their characteristics and their infection age.

There is no further generality in considering 
$\cY^N=\{Y_i^N\}_{i\le N}$ be 
an i.i.d. sequence of random variables  taking values in $\bY$, where $\bY$ is a Polish space,
by adjusting the definition of  $\wht \lambda$, see \cite[Theorem~3.19]{Ka02}.
\end{rem}

\begin{hyp} \label{hyp-lambda}
There exists a global upper-bound $\lambda^*<\infty$
	on the random functions  $\lambda^N_i$ taking values in $\bD(\RR_+, \RR_+)$,
	in the sense that the following holds almost surely: 
\begin{equation*}
\lambda^N_i(a) \le \lambda^*\,,
	\quad \forall N,\;  \forall i\in \II{1, N},\;  \forall a  \in \RR_+\,. 
\end{equation*}
\end{hyp}
\begin{rem}
For Assumption \ref{hyp-lambda} to hold under the formalism presented in Remark~\ref{rem_wht_lam},
it is sufficient to require the function $\wht \lambda$ to be uniformly upper-bounded by this value $\lambda^*$.
\end{rem}

The infection duration $\eta^N_i$ of individual $i\in \cS^N(0)$ is defined as 
\begin{equation*}
\eta^N_i := \sup\{a>0:  \lambda_i^N(a)>0\}\,.
\end{equation*}
Likewise, we consider the infection duration $\eta^{N, 0}_j$ of  
an initially infected individual $j\in \cI^N(0)$
with initial infection age $A^N(0)$:
\begin{equation*}
\eta^{N, 0}_j := \sup\{a>0:  \lambda_j^N(A^N(0) + a)>0\}\,,
\end{equation*}
while assuming that there is no pathological situation where the above set would be empty. 
	We specify the individual distribution of infection duration
through its cumulative distribution function. 
Under the formalism of Remark~\ref{rem_wht_lam}
it is to be thought as 
	depending only on the variables $(X_i^N, Y^N_i)$.

	\begin{hyp}\label{hyp-eta}
	There exists a measurable function $(F^c_x(a))_{x\in \bX, a\in \RR_+}$
		with values in $[0, 1]$
		which is right-continuous, non-increasing in the variable `$a$'
such that $\lim_{a\to \infty} F^c_x(a) = 0$ for all $x\in \bX$
		and  such that the two following identities 
		hold a.s. for any $N\ge 1$, $t>0$, $i\in \cS^N(0)$
		and $j\in \cI^N(0)$: 
\begin{equation*}
\begin{split}
\PR_0^N(\eta^N_i > t) 
&= F^{c}_{X^N_i}(t),
\\
\PR_0^N(\eta^{N, 0}_j > t)
&= \frac{F_{X^N_j}^{c}(A^N_j(0)+t)}{F_{X^N_j}^{c}(A^N_j(0))}\,. 
\end{split}
\end{equation*}
\end{hyp}


Recall that $\PR_0^N(\eta^N_i > t)$ equals by definition
$\PR\big(\eta^N_i > t \, \big| \, \cF_0^N\big)$
and similarly that  $\PR_0^N(\eta^{N, 0}_j > t)
=\PR\big(\eta^{N, 0}_j > t \, \big| \, \cF_0^N\big)$. 
Note that the condition on  $\eta^{N, 0}_j$ in Assumption \ref{hyp-eta} excludes the case where 
 the denominator in the above expression would be zero.  

 Let $F_x:= 1- F^c_x$, which is interpreted as the c.d.f. of the infection durations $\eta^N_i$ given that $X_i^N=x$.

\begin{rem}\label{rem_g}
This assumption appears more explicitly stated 
under the formalism presented in Remark~\ref{rem_wht_lam}.
By definition, 
there exists a deterministic measurable function 
$g: \bX\times [0, 1]\to \RR_+ $ 
such that $\eta^N_i = g(X^N_i, Y^N_i)$ a.s.
For each $x\in \bX$, we define the cumulative distribution: 
\[
F_x(t) = \PR(g(x, Y^N_1) \le t ) = \int_{[0, 1]} \idc{g(x,y) \le t} \Rd y
\]
and  set $F^c_x(\cdot) = 1- F_x(\cdot)$. 
$F^c_x$ is clearly non-increasing for any $x\in \bX$, with values in $[0, 1]$.
As a consequence, we check
the first identity in Assumption \ref{hyp-eta},
for any $N\ge 1$,  any $i\in \cS^N(0)$, and any $t>0$: 
\begin{equation*}
	\PR_0^N(\eta^N_i > t)
	=  \PR\big(g(X_i^N, Y^N_i) > t \, \big| \, X_i^N\big)  = F^c_{X^N_i}(t).
\end{equation*}
For any $N$ and any initially infected individual $j\in \cI^N(0)$,
the constraint that $\eta^{N, 0}_j$ is a well-defined positive quantity, which translates into $g(X_j^N, Y^N_j) > A_j^N(0)$,
explains the form of the second identity in Assumption~\ref{hyp-eta},
where 
for any $t>0$, 
\begin{equation*}
	\begin{split}
		\PR_0^N(\eta^{N,0}_j > t)
		&= \PR\big(g(X_j^N, Y^N_j) > t + A_j^N(0)\, \big| \, X_j^N, A_j^N(0), \idc{g(X_j^N, Y^N_j) > A_j^N(0)}\big)
		\\&= \frac{F_{X^N_j}^c(A^N_j(0)+t)}{F_{X^N_j}^c(A^N_j(0))}\,.
	\end{split}
\end{equation*}
\end{rem}

\begin{hyp}\label{hyp-Blambda}
There exists a deterministic measurable function $\bar\lambda$ from $\bX\times \RR_+$ to $\RR_+$
such that the two following identities hold for any $N\ge 1$, 	any $a\ge 0$, any $t\ge 0$, any $i\in \cS^N(0)$
and any $j\in \cI^N(0)$:
\begin{equation*}
\begin{split}
	\bE_0^N\big[\lambda^N_i(a)\big]
&= \bar\lambda(X_i^N, a),
\\ \bE_0^N\big[\lambda^{N, 0}_j(t)\big]
&=  \frac{\bar\lambda(X_j^N, A^N_j(0)+t) } {F^{c}_{X^N_j}(A^N_j(0))},
\end{split}
\end{equation*}
where $F^{c}_{X^N_j}(A^N_j(0))>0$ 
holds a.s. for any $j\in \cI^{N}(0)$. 
Moreover, 
the function  $\chi: \bX\times \RR_+\mapsto 
\RR_+$
defined as $\chi(x, a) = \bar\lambda(x, a)/F^{c}_{x}(a)$ 
if $F^{c}_{x}(a)>0$ and $\chi(x, a) = 0$ otherwise, 
is upper-bounded by $\lambda^*$.
\end{hyp}

Additional conditions on the regularity of $\bar\lambda$ will be stated in Assumption~\ref{hyp-w}, once the limiting measures on $\bX$ will be defined. 
Recall that $\bE_0^N\big[\lambda^N_i(a)\big]$
	equals by definition $\bE\big[\lambda^N_i(a)\, \big| \, \cF^N_0\big]$
	and similarly for $\bE_0^N\big[\lambda^{N, 0}_j(t)\big]$. 

\begin{rem}
Under the formalism presented in Remark~\ref{rem_wht_lam},
$\bar\lambda$ is directly expressed as follows:
\begin{equation*} 
\bar\lambda(x, a) =  \int_{[0, 1]} \wht \lambda(x, y, a)\; \Rd y\,,
\end{equation*}
and the first identity in Assumption \ref{hyp-Blambda} is automatic.
We can also check the second identity:
\begin{equation*}
\begin{split}
	\bE_0^N\big[\lambda^{N, 0}_j(t)\big]
&= \bE\big[\wht \lambda(X_j^N, Y_j^N, t + A_j^N(0))\, \big| \, X_j^N, A^N_j(0), \idc{g(X_j^N, Y^N_j) > A_j^N(0)}\big]
\\&=   \frac{\int_{[0, 1]} \wht \lambda(X_j^N, y, t + A_j^N(0))\; \Rd y} {F_{X^N_j}^c(A^N_j(0))}
= \frac{\bar\lambda(X_j^N, t+A^N_j(0)) } {F^{c}_{X^N_j}(A^N_j(0))},
\end{split}
\end{equation*}
where we have exploited in the second line the fact that $\wht \lambda(X_j^N, y, t + A_j^N(0))>0$
entails $g(X_j^N, Y^N_j) > t+ A_j^N(0)$ and a fortiori $g(X_j^N, Y^N_j) > A_j^N(0)$.
In addition, for any $x\in \bX$ and any $a\ge 0$ such that $F^c_x(a) >0$, since $\wht \lambda(x,y,a)=0$ on the set $\{g(x,y) \le a\}$, 
\begin{equation*}
\frac{\bar\lambda(x, a)}{F^c_x(a)} 
= \frac{\int_{[0, 1]} \wht \lambda(x, y, a)\, \idc{g(x, y) > a}\; \Rd y}{
\int_{[0, 1]}  \idc{g(x, y) > a}\; \Rd y}\,,
\end{equation*}
and is actually upper-bounded by $\lambda^*$
as soon as  $\wht \lambda$ is itself upper-bounded by $\lambda^*$.
\end{rem}

\begin{rem}\label{rem_lamb_dec}
In this work, we generally allow for intricate dependencies 
between the possible value of $\lambda_i^N(a)$ and the infection duration of individual $i$.
Though, 
a classical choice for the function $\wht \lambda$ is to be defined
in terms of two   given deterministic positive functions $\wtd\lambda(x, a)$ and $g(x,y)$
as follows: $\wht \lambda(x, y, a) = \wtd\lambda(x, a) \idc{a<g(x,y)}$.
In this case, 
the expression for $\bar\lambda$ simplifies as follows:
\begin{equation}\label{eq_wtd_lam}
\bar\lambda(X_i^N, a)  =  \wtd\lambda(X_i^N, a)\cdot F^c_{X^N_i}(a). 
\end{equation}
\end{rem}

\begin{rem}
	As an  instance of the case presented in the above Remark \ref{rem_lamb_dec}, 
the most standard Markovian description of a constant infectivity $\check \lambda>0$
with an exponential random duration with mean $\check g(x)$ is obtained as follows.
We can choose $\mu_Y$ to be the law of a standard exponential random variable with mean $1$,
on $\bY = \RR_+$,
and take  $\wtd\lambda(x,a) \equiv \check \lambda$ and $g(x,y) = \check g(x) \cdot y$.
The expression for $\bar\lambda$ further simplifies:
\[
\bar\lambda(X_i^N, a)  =  \check\lambda\cdot \exp\Big(-a/\check g(X^N_i)\Big). 
\]
\end{rem}

\begin{rem}
In the previous assumptions, we did not allow the infectious period to extend indefinitely, in line with the SIR formalism. It does however  not appear crucial in the following proofs. Some SI formalism without recovery 
or even a mix between SI and SIR formalism 
could be treated very similarly.
We mean situations where 
$\PR(\eta^N_i= +\infty)>0$ for some $i$,
with possibly $F_x$ upper-bounded by a constant strictly smaller than 1  for some non-negligible set of $x\in \bX$  
and the function $g$ in Remark~\ref{rem_g} with possibly infinite values.
\end{rem}

\subsection{Construction of the connectivity graph}

We consider the underlying connectivity graph among the individuals, which is denoted as a graph $(\cG^N, \cE^N)$ with $\cG^N$ being the set of $N$ nodes and  $\cE^N$ being the set of edges
{(which are undirected)}. 
We write $i \stackrel{N}{\sim} j$ to indicate that nodes $i$ and $j$ are connected, whose connectivity depends on $N$. 
We assume that the connectivity probability of nodes $i,j \in \cG^N$ is given by
 the  deterministic symmetric measurable function $\kappa^N: \bX\times \bX\to [0,1]$:
\begin{equation}\label{eq_def_kN}
\PR_0^N\big(i \stackrel{N}{\sim} j\big)
= \PR\big(i \stackrel{N}{\sim} j\, \big| \, \cF_0^N\big)
= \kappa^N(X^N_i, X^N_j),
\end{equation}
in terms of the  characteristic variables $X_i$ and $X_j$ for individuals $i$ and $j$.
We assume in addition that the events $\{i \stackrel{N}{\sim} j\}_{i\neq j}$
are mutually independent
and independent of the sequence $(Y_i^N)_{i\le N}$
conditionally on $\cF_0^N$.
\medskip

\begin{rem}\label{rem_kxy}
The typical construction of such a random graph starts from a graphon kernel. 
For example, 
assuming here that $\bX = [0, 1]$, 
let $\bar \kappa(x,x')$ be a deterministic function on $[0,1]^2$, 
such as $\bar \kappa(x,x') = x\cdot x'$.
One can sample  a $N$-node random graph with $\kappa^N = \bar\kappa$
so that $\PR_0^N\big(i \stackrel{N}{\sim} j\big) = \bar \kappa(X^N_i, X^N_j)$ for nodes $i,j=1,\dots, N$, where the $X^N_i$ are 
uniformly and independently distributed in $[0, 1]$.
Another very natural choice of $X^N_i$ is also $X^N_i=i/N$ for $i=1,\dots,N$,
which also leads to the limiting empirical measure $\bar \mu_X$, as defined in the  formula \eqref{eq_def_bmuX} below, being the Lebesgue measure on $[0, 1]$.
\end{rem}

\begin{rem}
In the case $\bar \kappa(x,x') = x\cdot x'$,
the degree of individual $i$ in the graph
then scales linearly with $N$ and with the individual type $X_i^N$.
The fact that the degree scales linearly in $N$ is typically how a ``dense" graph is defined.
Yet, we allow for more generality in the graph density 
with a sequence $(\kappa^N)$ of functions that may scale with $N$.
The first way to do so 
is to introduce a scaling factor $\eps^N>0$,
typically $\eps^N = N^{-\alpha}$ with $\alpha \in (0, 1)$,
so that $\kappa^N = \eps^N\cdot \bar \kappa$ as in \cite{KHT2022}.
Following \cite{DFG+24},
this form of scaling is however not assumed in order to allow for even more general
dependencies between the pair of individual characteristics $(X_i^N, X_j^N)$ and $N$.
The role of the denseness of the graph will be discussed
in Remark~\ref{rem_Var}, and in Section~\ref{sec_disc_MT}
after the proofs.
\end{rem}
 
\subsection{Force of infection}

To describe the force of infection, we introduce a random weight function 
$(\omega^N(i, j))_{i, j\in \II{1, N}}$
that is non-negative 
and equal to zero 
except for the edges $(i, j)$ of $\cE^N$.
$\omega^N(i, j)$ is meant to represent the scaling factor
	that translates the infectivity value of individual $j$,
	typically $\lambda_j^N(A_j^N(t))$ at time $t$, 
into the rate at which individual $i$ is subject
to an infectious contact with individual $j$,
which becomes $\omega^N(i, j) \cdot \lambda_j^N(A_j^N(t))$ at time $t$.

\begin{rem}
Heterogeneity	 in the susceptibility to infection
can be included in the model through these values $\omega^N(i, j)$,
such heterogeneity between individuals 
being possibly partly explained by the characteristic $X_i^N$.
This is a major reason for us not to assume that $\omega^N$ is symmetric.
Such a framework for $\omega^N(i, j)$ allows in addition
more realistic and intricate relations between the contact rates 
of individuals $i$ and $j$ 
depending on their respective characteristics $X_i^N$ and $X_j^N$.
Even more generally than
contact matrices typically exploited for epidemiological predictions,
see, e.g., \cite{MHJ+08},
we allow for the interplay between  characteristics that are continuous
and for an additional degree of independent randomness.

\end{rem}

\begin{hyp}\label{hyp_gamma}
The values $(\omega^N(i, j))_{i, j\in \II{1, N}}$ 
are mutually independent between different edges
and independent of the $(\lambda_i^N)_{i\le N}$ conditionally on $\cF_0^N$. 
There exists a deterministic measurable function $\gamma^N: \bX\times \bX\to \RR_+$
such that:
\begin{equation}
	\bE\big[\omega^N(i, j) \, \big| \, \cF_0^N,  \cE^N\big] 
	= 
	\begin{cases}
		\gamma^N(X_i^N, X_j^N) \quad \text{ if } (i, j) \in \cE^N,
		\\
		0 \hcm{2.5} \text{ otherwise.} 
	\end{cases}
	\label{eq_def_gamN}
\end{equation}
\end{hyp}

\begin{rem}
Under the formalism given in Remark~\ref{rem_wht_lam},
the values $(\omega^N(i, j))_{i, j\in \II{1, N}}$ are  independent of the $(\lambda_i^N)_{i\le N}$
if they are independent of the $(Y_i^N)$,  conditionally on $\cF_0^N$.
\end{rem}
The function $\gamma^N$
captures the expected infectious rate of interaction for active contacts. 
As a consequence of these definitions, we have 
\begin{equation}\label{eq_prop_omN}
	N\cdot\bE_0^N\big[ \omega^N(i, j)] = 
	\bar \omega^N(X_i^N, X_j^N),
	\quad \text{where }\,\, \bar\omega^N := N\cdot \kappa^N\cdot\gamma^N\,. 
\end{equation}
Remark that $\gamma^N$, thus $\bar \omega^N$, are not required to be symmetric,
given that the allowed asymmetry of $\omega^N(i, j)$
has no good reason to vanish after taking conditional expectation. 
However, recall that the construction imposes $\kappa^N$ to be symmetric. 
\medskip

We  also allow for a degree of variability 
in the infectious rate of interaction,
that we synthetize with the following variance term $\upsilon^N$,
for any $(i, j) \in \cE^N$:
\begin{equation}
	\upsilon^N(i, j)
	= \Var\big[\omega^N(i, j) \, \big| \, \cF_0^N,  \cE^N \big]\,.
\label{eq_def_upsN}
\end{equation}
Our upcoming Assumption~\ref{hyp-Var} entails that the quantity $\upsilon^N(i, j)$ is finite, for any $i, j\in \II{1, N}$. 
Given the above definition of $\gamma^N$, 
this  definition of $\upsilon^N$ shall be interpreted as follows:
\begin{equation}\label{eq_prop_ups}
\bE \big[\omega^N(i, j)^2 \, \big| \, \cF_0^N,  \cE^N \big]
= 
\begin{cases}
\upsilon^N(i, j) + \gamma^N(X_i^N, X_j^N)^2 \quad \text{ if } (i, j) \in \cE^N,
	\\
	0 \hcm{4.5} \text{ otherwise.} 
\end{cases}
\end{equation}

For each individual $i$,  the aggregated force of infection at time $t$  acting upon~$i$ is given by
\begin{equation}\label{eq_def_mfFiN}
\overline\mfF_i^N(t) 
= \sum_{j\in \cI^N(t)} \omega^N(i, j)\cdot \lambda^N_j(A^N_j(t))\,. 
\end{equation}

\subsection{Epidemic dynamics}
For each $i \in \cS^N(0)$, we describe the progression of the disease through the following process:
\begin{equation}\label{eq_def_DiN}
D_i^N(t) = \int_0^t \int_0^\infty \idc{D_i^N(s^-)=0} \idc{u \le  \overline\mfF_i^N(s^-)} Q_i(\Rd s, \Rd u)\,,
\end{equation}
where  $Q_i(\Rd s, \Rd u)$ is a standard Poisson random measure on $\RR_+^2$ with mean measure $\Rd s\, \Rd u$,
the $(Q_j)_{j\in \mathbb{N}}$ being mutually independent and independent of  $\cF_0^N$, 
the random graph generation $(\omega^N(j, k))_{j, k\le N}$,
and the random infectivity functions $(\lambda_j^N)_{j\le N}$
(in the sense that they are independent of the $\cY^N$ in the formulation \eqref{eq_def_lamY}).
Let
\begin{equation}
\tau_i^N = \inf\{t\ge 0;\, D_i^N(t)= 1\};\quad 
\sD^N(t) := \{i \in \cS ^N(0);\, \tau_i^N \le t\}, 
\label{eq_def_DN}
\end{equation}
so that the infection time $\tau_i^N$ is the unique jump time of $D_i^N$ 
(which takes the value $\infty$ in the absence of such a jump),
while $\sD^N(t)$ is the subset of initially susceptible individuals infected by the disease  by time $t$ after time $0$, 
or equivalently those $i\in \cS^N(0)$ 
	such that  $D_i^N(t) = 1$.
	In other words, $\sD^N(t) =( \cI^N(t) \cup \cR^N(t)) \setminus (\cI^N(0) \cup \cR^N(0))$.

	Using  $D_i^N(t)$ in \eqref{eq_def_DiN}, 
we also obtain the following alternative expression for  $\mfF_i^N(t)$ for each $i=1,\dots, N$: 
\begin{align} \label{eq_def_mfFiN2}
 \overline\mfF_i^N(t) 
&= \sum_{k \in \mathcal{I}^N(0)}  \omega^N(i, k)\lambda^N_k(A_k^N(0) + t)  \nonumber \\
 &\quad + 
 \sum_{j\in  \cS^N(0)} \omega^N(i, j)  \int_0^t  \int_0^\infty \ \lambda^N_j(t-s) \idc{D^N_j(s^-)=0} \idc{u \le \overline \mfF_j^N(s^-)}  Q_j(\Rd s, \Rd u)\,.
\end{align}

Define the following measure-valued processes associated with the susceptible, infectious and recovered individuals: 
  \begin{equation}
 	\bar\mu^{S, N}_t(\Rd x)
 	= \frac1N\sum_{i\in \cS^{N}(t)} \delta_{X^N_i}(\Rd x)\,,
 	\label{eq_def_muSNt}
 \end{equation}
 \begin{equation}
	\bar\mu^{I, N}_t(\Rd x, \Rd a)
	=  \frac1N\sum_{i\in \cI^{N}(t)} \delta_{X^N_i}(\Rd x)\delta_{A_i^N(t)}(\Rd a)\,,
	\label{eq_def_muINt}
\end{equation}
\begin{equation}
	\bar\mu^{R, N}_t(\Rd x)
	= \frac1N \sum_{i\in \cR^{N}(t)} \delta_{X^N_i}(\Rd x)\,. 
	\label{eq_def_muRNt}
\end{equation}

For any $t>0$, 
	 	$\bar\mu^{S, N}_t$ and $\bar\mu^{R, N}_t$
are regarded as elements in the set $\cM(\bX)$
	 	of non-negative finite measure on $\bX$,
	which is  equipped with the topology of weak convergence,
	 	and similarly for $\bar\mu^{I, N}_t$
	 belonging 	to the set $\cM(\bX\times \RR_+)$. 
 
The dynamics of these measure-valued processes 
can be represented using  $D_i^N(t)$ and $ \overline\mfF_i^N(t)$ as follows, 
in terms of test functions  $\varphi \in C_b(\bX)$ and $\psi \in C_b(\bX\times \RR_+)$: 
\begin{itemize}
\item for the susceptible process:
\begin{equation} \label{eq_def_muSNt1}
	\langle\bar\mu^{S,N}_t, \varphi\rangle 
	=\langle\bar\mu^{S,N}_0, \varphi\rangle - 
	\frac1N\sum_{i\in  \sD^N(t)} \varphi(X^N_i)\,,
\end{equation}
with 
	\begin{equation*}
		\sum_{i\in  \sD^N(t)} \varphi(X^N_i)
		=  \sum_{i\in  \cS^N(0)} \int_0^t  \int_0^\infty \idc{D_i^N(s^-)=0} \idc{u \le \overline \mfF^N_i(s^-)} \varphi(X^N_i) Q_i(\Rd s, \Rd u), 
	\end{equation*}
\item for the infected process:
\begin{equation}	\label{eq_def_muINt1}
	\langle\bar\mu^{I,N}_t, \psi\rangle 
	= \frac1N\sum_{j\in \cI^N(0)} \idc{\eta_j^{N, 0} >t} \psi(X^N_j, A_j^N(0)+t)
	+ \frac1N\sum_{i\in \sD^N(t)} \idc{\tau^N_i +\eta^N_i>t} \psi(X^N_i, t-\tau^N_i)\,,
\end{equation}
with 
	\begin{multline*}
	\sum_{i\in \sD^N(t)} \idc{\tau^N_i +\eta^N_i>t} \psi(X^N_i, t-\tau^N_i)
	\\	=  \sum_{i\in  \cS^N(0)} \int_0^t  \int_0^\infty 
		\idc{D_i^N(s^-)=0} \idc{u \le \overline \mfF^N_i(s^-)} \idc{\eta^N_i>t-s} \psi(X^N_i, t - s) Q_i(\Rd s, \Rd u)\,,
	\end{multline*}
\item for the recovered process:
\begin{equation}	\label{eq_def_muRNt1}
	\langle\bar\mu^{R,N}_t, \varphi\rangle =  \langle \bar\mu^{R,N}_0, \varphi \rangle  
	+ \frac1N\sum_{j\in \cI^N(0)} \idc{\eta_j^{N, 0} \le t} \varphi(X^N_j) 
	+ \frac1N\sum_{i\in \sD^N(t)} \idc{\tau^N_i +\eta_i^N\le t} \varphi(X^N_i)\,.
\end{equation}
\end{itemize}

Note that the processes $S^N(t)$,  $I^N(t)$ and  $R^N(t)$
	corresponding to respectively the numbers of susceptible, infected and recovered individuals at time $t$
	 can be obtained from these measure-valued processes: 
$S^N(t) = N\langle\bar\mu^{S,N}_t, \mathbbm{1}\rangle $,  $I^N(t) = N\langle\bar\mu^{I,N}_t, \mathbbm{1}\rangle$ and  $R^N(t) = N\langle\bar\mu^{R,N}_t, \mathbbm{1}\rangle$ with $\mathbbm{1}(x)\equiv 1$ and $\mathbbm{1}(a,x) \equiv 1$. 

We make the following assumptions on the initial quantities.
\begin{hyp} \label{hyp-initial}
There exist finite measures $\bar\mu^S_0(\Rd x)$ on $\bX$, $\bar\mu^I_0(\Rd x,\Rd a)$ on $\bX\times \RR_+$ and $\bar\mu^R_0(\Rd x)$ on $\bX$ 
that are the weak limits in probability as $N\to\infty$ of respectively $\bar \mu_0^{S, N}$, $\bar \mu_0^{I, N}$, and $\bar \mu_0^{R, N}$.
\end{hyp}

We consider also the two distributions 
$\bar\mu_X^N$ and $\bar\mu_X$
of characteristics on $\bX$ for the whole population: 
\begin{align}
\bar \mu_X^N(\Rd x)
&=  \bar\mu^{S, N}_0(\Rd x) + \langle \bar\mu^{I, N}_0(\Rd x, .), {\bf1}\rangle + \bar\mu^{R, N}_0(\Rd x)
= \frac{1}{N} \sum_{i\le N} \delta_{X^N_i}(\Rd x),\quad 
\label{eq_def_muXN}
\\
\bar \mu_X(\Rd x)
&= \bar\mu^S_0(\Rd x) + \langle \bar\mu^{I}_0(\Rd x, .), {\bf1}\rangle + \bar\mu^{R}_0(\Rd x).
\label{eq_def_bmuX}
\end{align}
Here ${\bf 1}(a) \equiv 1$.  
A direct consequence of Assumption \ref{hyp-initial}
is that $\bar \mu_X $ is a probability distribution, 
and  $\bar \mu_X^N$  converges weakly to $\bar \mu_X$.

In addition,
we make the following two assumptions 
to capture the behavior of the graph structure $(\omega^N(i, j))_{i, j\in \II{1, N}}$
as $N$ tends to infinity,
firstly in terms of conditional expectations.
We include additional regularity conditions together with the first assumption,
on the quantities introduced in Assumptions~\ref{hyp-eta},
	 \ref{hyp-Blambda} and \ref{hyp-initial}. 
\begin{hyp} \label{hyp-w}
Recalling the relation $\bar\omega^N = N\cdot \kappa^N\cdot\gamma^N$,
	the following convergence as $N\to\infty$
	holds in probability uniformly over $\bX\times\bX$:
	\[
	\bar\omega^N  \to \bar \omega, 
	\]
where $\bar\omega: \bX\times \bX\to \RR_+$
is some deterministic and bounded measurable function 
	that is $\bar \mu_X^{\otimes 2}$-almost everywhere (a.e.) continuous.

$\bar \mu^I_0$ is absolutely continuous with respect to $(\bar \mu_X\otimes Leb)$,
where $Leb$ denotes the Lebesgue measure on $\RR_+$.
The function $:(x, a)\mapsto F_x(a)$ 
is $\bar \mu^I_0$-a.e. continuous
while the function $\bar\lambda$
is $(\bar \mu_X\otimes Leb)$-a.e. continuous.
\end{hyp}

Assumption~\ref{hyp-w} is stated in terms of 
 the convergence of the product $N\gamma^N\kappa^N$
without requiring separate convergence properties 
of $\gamma^N$  and  $\kappa^N$. 
This choice was made to emphasize 
that the scaling in $N$ between the contact probability and the contact rate
could itself be mediated by various kinds of interactions
reflected through the individual types $X_i^N$ and $X_j^N$, 
for instance with more denseness but less frequent contacts 
for the infections during travels than at the workplaces.
\smallskip

Regarding our next convergence results in probability, there is no loss of generality in considering the following assumption.
\begin{hyp}\label{hyp_ws}
There exists $\omega^*>0$
that is an upper-bound of $\bar\omega^N$
uniformly on $\bX^2$ and on  $N$.
\end{hyp}

The next assumption provides the crucial estimate 
to deal with the variability of the random graph generation.
Its relation to the denseness of the graph
	and to possibly high levels of infection rates
	is discussed after the statement of the main result.
\begin{hyp} \label{hyp-Var}
	The two following quantities $\bar \gamma^N$ and $\Upsilon^N$ 
	converge in probability to zero as $N\to\infty$:
	\begin{equation}\label{eq_def_bGammaN}
		\begin{split}
			&\bar \gamma^N := \frac 1 {N^2} \sum_{i, j} \gamma^N(X_i^N, X_j^N),
			\\
			&\Upsilon^N := \frac 1 {N} \sum_{i, j} \bE_0^N\big[\upsilon^N(i, j); (i, j)\in \cE^N\big]
			\\&\qquad = \frac 1 {N} \sum_{i, j} \kappa^N(X_i^N, X_j^N)\cdot \bE_0^N\big[\upsilon^N(i, j) \, \big| \, (i, j)\in \cE^N\big]\,.
		\end{split}
	\end{equation}
\end{hyp}

By $\bE_0^N\big[\upsilon^N(i, j) \big| (i, j)\in \cE^N\big]$ is meant the ratio of $\bE\big[\upsilon^N(i, j); (i, j)\in \cE^N\, \big| \, \cF_0^N\big]$ divided by $\PR\big[(i, j)\in \cE^N\, \big| \, \cF_0^N\big]= \kappa^N(X_i^N, X_j^N)$.

\begin{rem}\label{rem_Var}
		To better understand the variance term in Assumption \ref{hyp-Var},
		let us consider the case 
		where the heterogeneity in individual contacts
		is purely neutral, so independent of the $(X_i^N)$.
	Let us 	assume  the existence of three positive parameters 
	$\bar\kappa^N$, $\bar\gamma^N$ and $\bar\sigma^N \in (0, \bar\gamma^N)$
such that the two following conditions hold 
in addition to the above-described construction:
		\begin{enumerate}
			\item [(i)] the graph of active contacts is an Erdös-R{\'e}nyi graph with parameter $\bar\kappa^N$, that is 
			$\PR_0^N\big(i \stackrel{N}{\sim} j\big) = \bar\kappa^N$ for any $i, j$.
			\item [(ii)] on the event	$\{(i, j)\in \cE^N\}$
			and conditionally on $\cF_0^N$,
			$\omega^N(i, j)$ is distributed as a uniform random variable 
			between $\bar\gamma^N - \bar\sigma^N$ and $\bar\gamma^N + \bar\sigma^N$.
		\end{enumerate}
		Then \eqref{eq_def_kN} and \eqref{eq_def_gamN} are satisfied with 
		\begin{equation*}
			\kappa^N(x, y) \equiv \bar\kappa^N,\quad \gamma^N(x, y) \equiv \bar\gamma^N,\quad
			\upsilon^N(i, j) \equiv \frac{(\bar \sigma^N)^2}{3}.
		\end{equation*}
		Assumption \ref{hyp-w} then translates into the convergence of the product $	N\,\bar\kappa^N\, \bar\gamma^N$
		to some limiting value $\bar \omega$.
		The notation $\bar\gamma^N$ is coherent with the formula given in Assumption~\ref{hyp-Var}
		while $\Upsilon^N = N\, \bar \upsilon^N\, \bar \kappa^N
		= N\cdot (\bar \sigma^N)^2\cdot \bar \kappa^N/3$.
		Let us assume the system to be non-degenerate in that $\bar \omega>0$ and $\bar \gamma^N>0$ for any $N$.

		Then, $\Upsilon^N \sim \bar \omega  \cdot (\bar \sigma^N)^2/ (3 \gamma^N)$
		converges to zero as required
		if and only if $(\bar \sigma^N)^2 \ll \gamma_N$.
		Since $\bar \sigma^N < \gamma_N$ in this case (to keep $\omega^N$ non-negative),
		both properties hold true if and only if $\gamma_N$ tends to zero, 
		which is exactly the first condition in Assumption~\ref{hyp-Var}.
		So we do not have any additional restriction on $\bar \sigma_N$ in this model.
		
				Generally with $\bar \upsilon^N$ the variance of the corresponding distribution in $(ii)$,
		$\Upsilon^N$ converges to zero 
		if and only if the standard deviation $\sqrt{\bar \upsilon^N}$ 
		is negligible against the root of the average rate 
		$\sqrt{\bar \gamma^N}$.
Given that $\bar \gamma^N$
itself is expected to tend to 0,
we stress that our result covers
 fluctuation levels
 in the rate of transmission
that are quite large in comparison to the expected value.
		\end{rem}

\section{Functional law of large numbers} \label{sec:FLLN}

For brevity, we use the notations
$\bD_1 := \bD(\RR_+, \cM(\bX))$ and $\bD_2 := \bD(\RR_+, \cM(\bX\times \RR_+))$.

\begin{theo}
	\label{thm_cvg}
Under Assumptions~\ref{hyp-lambda}--\ref{hyp-Var},
\[
(\bar \mu^{S, N}_{\cdot}, \bar \mu^{I, N}_{\cdot}, \bar \mu^{R, N}_{\cdot})
\to (\bar \mu^{S}_{\cdot}, \bar \mu^{I}_{\cdot}, \bar \mu^{R}_{\cdot})
\]
in $\bD_1 \times \bD_2 \times \bD_1$ as $N\to\infty$.
In the above limit, 
 $\bar \mu^{S}_{\cdot}$ is the first component 
 of the unique solution $(\bar \mu^{S}_{\cdot}, \overline{\mfF}(\cdot))$
 to the following set of equations, 
 \begin{align} \label{eq_def_muSt1}
\bar\mu^{S}_t(\Rd x) = \bar\mu^{S}_0(\Rd x) - \int_0^t  \overline{\mfF}(s,x) \bar\mu^{S}_s(\Rd x) \Rd s\,,
 \end{align}
 and
 \begin{equation}\label{eq_def_bar_mfFtx}
 \begin{split} 
\overline{\mfF}(t,x) &= \int_0^\infty\hcm{-0.3} \int_{\bX}  \bar \omega(x,x') \frac{\bar\lambda(x', a'+t)}{F^c_{x'}(a')}    \bar\mu_0^I(\Rd x', da')  \\
& \quad + \int_0^t\hcm{-0.2} \int_{\bX} \bar \omega(x,x') \bar\lambda(x',t-s) \overline{\mfF}(s,x') \bar\mu^S_s(\Rd x')\Rd s  \,,
 \end{split}
 \end{equation}
 with  $\bar\lambda$, $\bar \omega$ being given respectively in Assumptions \ref{hyp-Blambda} and \ref{hyp-w}.
The uniqueness of the above system
	is stated among the potential candidates $(\bar \mu^{S}_{t}, \overline{\mfF}(t))_{t\ge 0}$
in which $\bar \mu^{S}_{\cdot}$ is a $\cM(\bX)$-valued
càd-làg process 
such that $\mu_t(\bX)\le 1$
for any $t\ge 0$,
while $\overline{\mfF}(\cdot)$ is a càd-làg process whose values 
are bounded measurable functions from $\bX$ to $\RR_+$,
with local in time upper-bounds.

Given the pair $(\bar\mu^{S}_{\cdot}, \overline{\mfF}(\cdot))$
and the initial conditions $\bar{\mu}^I_0$ and $\bar{\mu}^R_0$, the explicit formulas for $\langle\bar\mu^I_t,\psi\rangle$ and $\langle\bar{\mu}^R_t,\phi\rangle$ 
 with the test functions  
$\psi \in C_b(\bX\times \RR_+)$ and $\varphi \in C_b(\bX)$	
	are as follows: 
   \begin{equation}
\begin{split}
	 	\langle \bar\mu^{I}_t, \psi\rangle  
 	&= \int_0^\infty\hcm{-0.3} \int_{\bX} \psi(x, a+t)  \frac{F^c_x(a+t)}{F^c_x(a)} \bar\mu_0^I(\Rd x,\Rd a) \\
 	& \quad + \int_0^t\hcm{-0.2} \int_{\bX} \psi(x,t-s)F^c_x(t-s) \overline{\mfF}(s,x) \bar\mu^S_s(\Rd x)\Rd s  \,,
\end{split}
\label{eq_def_bmuI}
 \end{equation}
 and 
   \begin{equation}
	\begin{split}
 	\langle \bar\mu^{R}_t, \varphi\rangle 
 	&= \langle \bar\mu^{R}_0, \varphi\rangle + 
 	\int_0^\infty\hcm{-0.3} \int_{\bX} \varphi(x)  \bigg(1-\frac{F^c_x(a+t)}{F^c_x(a)} \bigg)
 	\bar\mu_0^I(\Rd x,\Rd a) \\
 	&\quad  + \int_0^t\hcm{-0.2} \int_{\bX} \varphi(x) F_x(t-s) \overline{\mfF}(s,x) \bar\mu^S_s(\Rd x)\Rd s\,. 
 	\end{split} 
\label{eq_def_bmuR}
\end{equation}
\end{theo}

Note that the unique solution to equation~\eqref{eq_def_muSt1} is given by 
\begin{equation}\label{eq_def_bmuS2}
	\bar\mu^{S}_t (\Rd x) 
=  \,\exp\bigg(-\int_0^t  \overline{\mfF}(s,x) \Rd s\bigg) \bar\mu^{S}_0(\Rd x)\,,
\end{equation}
and that \eqref{eq_def_bmuI} results,  evaluated for any $x\in \bX$ 
	with the function 
	$\psi(x', a')= \omega(x, x')\frac{\bar \lambda(x', a')}{F_{x'}^c(a')}$
and after comparison with \eqref{eq_def_bar_mfFtx}, in the following expression for $\overline{\mfF}$:
	\begin{equation}\label{eq_prop_bar_mfFtx}
\overline{\mfF}(t,x)
= \int_0^\infty\hcm{-0.3} \int_{\bX}  \bar \omega(x,x') \frac{\bar\lambda(x', a')}{F^c_{x'}(a')}    \bar\mu_t^I(\Rd x', da')\,.
	\end{equation}

\begin{rem}
Our conditions on the limiting kernel of interaction $\bar \omega$ 
encompass most, if not all,
 of the (known) pairwise kernel interactions found in applications.
 It is through the expression of the limiting force of infection $\overline\mfF$
	in terms of the limiting kernel $\bar \omega$
	that we can relate our result
	to 	dynamics on a graphon.

	Consider the special case:  $\bar \omega(x, x') = \bar \gamma\cdot \bar\kappa(x, x')= \bar \gamma\cdot x\cdot x'$ for some value $\bar \gamma>0$. 
	This multiplicative form for $\bar \omega(x, x') $ is a typical example of a graphon. 
	Here we assume that the $(X_i^N)$ are uniformly distributed in $[0, 1]$, and $\bar \kappa^N = \bar \kappa$ for any $N$ and $\bar \kappa(x, x') = x\cdot x'$ as in Remark~\ref{rem_kxy}, and
	$\bar \gamma^N=\bar \gamma/N$ in Assumption~\ref{hyp-w} (hence,  $\gamma^N$ 
does not depend on the types of the individuals $i$ and $j$ in contact and $\Upsilon^N \equiv 0$).
This kernel structure displays 
a convenient property of separation 
between three contributions:
the role $(i)$  of the susceptible type, $(ii)$ of the infector type
and $(iii)$ of the contact rate.
This property
is then inherited by the force of infection
in that 
\[\overline{\mfF}(t,x) = \bar \gamma\cdot x\cdot\check{\mfF}(t)\]
where $\check{\mfF}(t)$ describes the aggregated contribution 
of the potential infectors at time $t$:
	\begin{equation*}
	\overline{\mfF}(t)
	= \int_0^\infty\hcm{-0.3} \int_{[0,1]}  x'\cdot \frac{\bar\lambda(x', a')}{F^c_{x'}(a')}    \bar\mu_t^I(\Rd x', da')\,.
	\end{equation*}
	
 \end{rem}

 	A reformulation of this system of integro-delay equations~\eqref{eq_def_muSt1}-\eqref{eq_def_bmuR} is proposed in Section~\ref{sec_PDE} 
	as a system of partial differential equations.

We also remark that   
in the special case of $\lambda_i^N(a)$ discussed in Remark \ref{rem_lamb_dec}, given the expression of $\bar\lambda(x,a)$ in \eqref{eq_wtd_lam}, we obtain the following:
	\begin{equation*}
\overline{\mfF}(t,x)
= \int_0^\infty\hcm{-0.3} \int_{\bX}  \bar \omega(x,x') \wtd\lambda(x', a')    \bar\mu_t^I(\Rd x', da')\,.
	\end{equation*}

\section{Properties of the limiting system of equations}  \label{sec:limit} 

\subsection{Existence and uniqueness of solution to the limiting equations}

\begin{prop}
	\label{prop_bmu_bF_UE}
Under Assumptions \ref{hyp-lambda} and \ref{hyp-w}, the set of equations \eqref{eq_def_muSt1}--\eqref{eq_def_bar_mfFtx} has a unique solution $(\bar\mu^{S}_{\cdot}, \overline\mfF_{\cdot})$.  
\end{prop}

\begin{proof}
We first prove the uniqueness. Before proceeding, we establish some useful bounds. 
Suppose that $(\bar\mu^{S}_{\cdot}, \overline{\mfF}_{\cdot})$
is a solution. 
By Assumptions \ref{hyp-lambda} and \ref{hyp-w}, we have $\bar\lambda(x, t) \le \lambda^*$
and $\bar \omega(x,x')  \le \omega^*$ for some $\omega^*>0$.
Recalling \eqref{eq_def_bar_mfFtx} and Assumption~\ref{hyp-Blambda}, we derive
\begin{equation}
	\begin{split}
		\overline{\mfF}(t,x)
		&\le \lambda^* \int_{\bX}  \bar \omega(x,x')    
		\int_0^\infty \bigg[ \frac{F^c_{x'}(a'+t)}{F^c_{x'}(a')} \bigg] \bar\mu_0^I(\Rd x', \Rd a') 
		\\
		& \quad + \lambda^* \int_{\bX} \bar \omega(x,x')
		\int_0^t  \overline{\mfF}(s,x')
		\exp\bigg(-\int_0^s \overline{\mfF}(r,x') \Rd r\bigg) \Rd s\;
		\bar\mu^S_0(\Rd x')  \,
		\\&\le \lambda^* \omega^*\int_{\bX} 
		 \Big[\langle \bar\mu_0^I(\Rd x', .), \mathbbm{1}\rangle
		+ \bar\mu^S_0(\Rd x')\Big] 
		\\&
		\le \lambda^*  \omega^* < \infty,
	\end{split}
	\label{eq_apr_mfF}
\end{equation}
where we exploited that 
1 is a natural upper-bound of  the integral over $s$ in the second line
and that $\bar \mu_X$ in \eqref{eq_def_bmuX}
is a probability measure.

On the other hand 
by  \eqref{eq_def_muSt1},
\begin{equation}
	\bar\mu^{S}_s(\Rd x') \le \bar\mu^S_0(\Rd x')\le \bar\mu_X(\Rd  x).
	\label{eq_apr_muS}
\end{equation}
	Suppose now that there are two solutions  $(\bar\mu^{S, \ell}_{\cdot}, \overline{\mfF}^{\ell}_{\cdot})$, $\ell=1,2$. 
Let $D_t = \Ninf{\overline{\mfF}^{1}(t, .)-\overline{\mfF}^{2}(t, .)}$.
By \eqref{eq_def_bar_mfFtk}:
\begin{equation}
	\begin{split}
		\overline{\mfF}^{1}(t,x)  -  \overline{\mfF}^{2}(t,x) 
		&= \int_0^t\hcm{-0.2} \int_{\bX} \bar \omega(x,x') \bar\lambda(x',t-s) \overline{\mfF}^{1}(s,x') (\bar\mu^{S,1}_s(\Rd x')- \bar\mu^{S,2}_s(\Rd x')) \Rd s \\
		& \qquad  + \int_0^t\hcm{-0.2} \int_{\bX} \bar \omega(x,x') \bar\lambda(x', t-s) (\overline{\mfF}^{1}(s,x')-\overline{\mfF}^{2}(s,x') ) \bar\mu^{S,2}_s(\Rd x')\Rd s \,. 
	\end{split}
	\label{eq_diff2_mfF}
\end{equation}
Since $:a\in \RR_+\mapsto e^{-a}$ is 1-Lipschitz continuous:
\begin{equation*}
	\Big\|\exp\bigg(-\int_0^s \overline{\mfF}^{1}(r,\,.) \Rd s\bigg)
	-  \exp\bigg(-\int_0^s \overline{\mfF}^2(r,\,.) \Rd r\bigg)\Big\|_\infty
	\le \int_0^s D_r\,\Rd r\,.
\end{equation*}
We then deduce from  \eqref{eq_def_bmuS2} that:
$ \left\| \bar\mu^{S,1}_s -\bar\mu^{S,2}_s\right\|_{TV}
\le \int_0^s D_r\,\Rd r$.
By injecting this inequality in \eqref{eq_diff2_mfF}
with \eqref{eq_apr_mfF} and \eqref{eq_apr_muS},
we deduce for any $T>0$ and $t\in [0, T]$
that $D_t
\le C \int_0^t D_s \Rd s$
where $C = \lambda^* \omega^* + (\lambda^* \omega^*)^2\, T$.
Thanks to Gronwall's inequality (for any $T>0$),
the proof of the  uniqueness is concluded.
\smallskip

Finally,  existence can be proved by the following modified Picard iteration. 
We initialize with $\bar\mu^{S, 0}_s(\Rd x')
= \bar\mu^{S}_0(\Rd x')$, $\overline{\mfF}^{0}(t,x)=0$,
then for any $k\ge 0$:
\begin{align} \label{eq_prop_muStk}
	\bar\mu^{S, k+1}_t(\Rd x) = \bar\mu^{S}_0(\Rd x) - \int_0^t  \overline{\mfF}^{k}(s,x) \bar\mu^{S, k+1}_s(\Rd x) \Rd s\,,
\end{align}
and
\begin{equation}\label{eq_def_bar_mfFtk}
	\begin{split} 
		\overline{\mfF}^{k+1}(t,x) &= \int_0^\infty\hcm{-0.3} \int_{\bX}  \bar \omega(x,x') \frac{\bar\lambda(x', a'+t)}{F^c_{x'}(a')}    \bar\mu_0^I(\Rd x', da')  \\
		& \quad + \int_0^t\hcm{-0.2} \int_{\bX} \bar \omega(x,x') \bar\lambda(x',t-s) \overline{\mfF}^{k}(s,x') \bar\mu^{S, k+1}_s(\Rd x')\Rd s  \,.
	\end{split}
\end{equation}
In these two definitions, we make  the choice of the product between $\overline{\mfF}^{k}$ and $\bar\mu^{S, k+1}$ because then we have the following explicit expression of $\bar\mu^{S, k+1}$ in terms of $\overline{\mfF}^{k}$:
 \begin{align} \label{eq_def_muStk}
	\bar\mu^{S, k+1}_t(\Rd x) = 
		\exp\bigg(-\int_0^t \overline{\mfF}^k(r,x) \Rd r\bigg) 
	\bar\mu^{S}_0(\Rd x)\,,
\end{align}
similarly to \eqref{eq_def_bmuS2}.
We can thus justify similarly the a priori estimates \eqref{eq_apr_muS} and \eqref{eq_apr_mfF}  for the above approximating sequence.
Then, the arguments used for proving uniqueness 
yield the convergence of our Picard iteration, exactly as in the usual case of globally Lipschitz coefficients.
\end{proof}

\subsection{Regularity of $\overline{\mfF}(\cdot)$}
\begin{lem}\label{lem_reg_bF}
The function $:x\mapsto \overline{\mfF}(., x)$ 
with values in the set of bounded measurable functions from $\RR_+$ to itself
equipped with the uniform norm topology
is $\bar \mu_X$-a.e. continuous.
\end{lem}

\begin{proof}
Recalling \eqref{eq_def_bar_mfFtx}
and Assumption~\ref{hyp-Blambda}, 
we first deduce the following inequality for any $t\ge 0$ and any $x_1, x_2\in \bX$:
\begin{multline*}
	\Big|\overline{\mfF}(t, x_1)-\overline{\mfF}(t, x_2)\Big|
\le \lambda^* \int_\bX |\bar \omega(x_1, x') - \bar \omega(x_2, x')|
\Big\{\Big(\frac{F^c_{x'}(a'+t)}{F^c_{x'}(a')}\Big) \bar\mu_0^I(\Rd x', \Rd a')
\\+  \int_0^t \overline{\mfF}(s, x') \exp\Big(-\int_0^s \overline{\mfF}(r, x') \Rd r\Big) \Rd s
\;\bar\mu_0^S(\Rd x')
\Big\}.
\end{multline*}
With similar arguments as in the proof of Proposition~\ref{prop_bmu_bF_UE}
to deduce \eqref{eq_apr_mfF},
notably with Assumption \ref{hyp-initial},
we obtain 
\begin{equation}
\Ninf{\overline{\mfF}(., x_1)-\overline{\mfF}(., x_2)}
\le \lambda^* \int_\bX |\bar \omega(x_1, x') - \bar \omega(x_2, x')|\, \bar \mu_X(\Rd x').
\label{eq_prop_DeltaF}
\end{equation}
As stated in Appendix~\ref{sec:appendix} in Lemma~\ref{lem_pair_ctn} 
and proved just afterwards,
the function $x\mapsto \bar \omega(x, .)$ is
$\bar \mu_X$ a.e. continuous 
with the $L^1(\bar \mu_X)$ distance,
since $\bar \omega$ is $\bar \mu_X ^{\otimes 2}$-a.e. continuous by virtue of Assumption \ref{hyp-w}. 
This concludes the proof of Lemma \ref{lem_reg_bF}.
\end{proof}

\subsection{Alternative representation of the limit as PDEs}\hfill\\
\label{sec_PDE}

Provided that the hazard rate function corresponding to the durations before remission
is regular enough, one can describe the solution $\bar\mu^{I}$
of the limiting system in Theorem~\ref{thm_cvg}
in terms of a PDE
as stated in the next proposition.
\begin{prop}\label{prop_PDE_eq}
	Assume that $F_x$ is absolutely continuous with density $f_x$ for each $x\in \bX$
	and that $F^c_x(a)>0$ for any $a\in \RR_+$.
	Assume that $h_x(a) = f_x(a)/F^c_x(a)$, the hazard rate function,
	is continuous and bounded  uniformly in both $x\in\bX$ and $a\in \RR_+$. 
	Assume moreover that $\bar \mu^I_0(\bX\times \{0\}) = 0$.
	Then, 
	$\bar\mu^{I}_t(\Rd x,\Rd a)$ in \eqref{eq_def_bmuI} is the unique solution to the following equation:
 for any $\psi \in C^{0,1}(\bX\times \RR_+)$ and $t>0$,
	\begin{equation} \label{eqn-barmuI-PDE-1}
		\frac{d}{dt} \langle \bar\mu^{I}_t, \psi\rangle =  \langle \bar\mu^{I}_t, \partial_a\psi- h\psi\rangle   + \int_{\bX} \psi(x,0) \overline{\mfF}(t,x) \bar\mu^S_t(\Rd x) \,. 
	\end{equation}
	Hence, $\bar\mu^{I}_t(\Rd x,\Rd a)$ is the unique solution to the following PDE:
	\begin{equation}\label{eqn-barmuI-PDE}
		\langle \partial_t \bar\mu^{I}_t, \psi\rangle  + \langle \partial_a \bar\mu^{I}_t, \psi\rangle  = -  \langle h\,\bar\mu^{I}_t, \psi\rangle 
	\end{equation} 
	with the initial condition $\bar\mu^{I}_0$ given in Assumption \ref{hyp-initial} and the boundary condition at $a=0$:  $\tilde\mu^I_t(\Rd x,0) = \overline{\mfF}(t,x) \bar\mu^S_t(\Rd x)$, where $\tilde\mu^I_t(\Rd x,a)$ is the ``density'' of $\bar\mu^I_t(\Rd x,\Rd a)$, that is, $\bar\mu^I_t(\Rd x,\Rd a) =\tilde\mu^I_t(\Rd x,a)\Rd a$ and $a\mapsto\tilde\mu^I_t(\Rd x,a) $ is continuous at $0^+$. 	
\end{prop}

Recalling \eqref{eq_def_bmuS2},
remark that we can also write the last term in \eqref{eqn-barmuI-PDE-1} as 
\[
\int_{\bX} \psi(x,0) \overline{\mfF}(t,x)  \exp\Big(-\int_0^t  \overline{\mfF}(s,x) \Rd s\Big) \bar\mu^{S}_0(\Rd x)\,. 
\]
which only involves the initial $\bar\mu^{S}_0(\Rd x)$ and $\overline{\mfF}(t,x)$. 

We directly identify from Proposition~\ref{prop_PDE_eq}
	and \eqref{eq_prop_bar_mfFtx}
	that the system of equations under consideration
	is a weak formulation of the following system.
	Here, $u^I_t(x,a)$ (resp. $u^S_t(x)$) 
	represent the density of $\bar \mu_t^I$ (resp. $\bar \mu_t^S$)
with respect to $\bar \mu_X(\Rd x)\, \Rd a$
(resp. $\bar \mu_X(\Rd x)$):
\begin{equation}\label{eq_lim_PDE_syst}
	\left\{
\begin{aligned}
&\partial_t u^I_t(x,a) + \partial_a u^I_t(x,a)
= -h(x, a)\, u^I_t(x,a)\,,
\\&u^I_t(x, 0) = - \partial_t u^S_t(x)
=\overline{\mfF}(t,x)\, u^S_t(x)\,,
\\&\overline{\mfF}(t,x)
= \int_0^\infty\hcm{-0.3} \int_{\bX}  \bar \omega(x,x') \frac{\bar\lambda(x', a')}{F^c_{x'}(a')}  
u^I_t(x',a')\, \bar \mu_X(\Rd x')\, \Rd a'\,.
\end{aligned}
\right.
\end{equation}
We can associate to this system the density process $u^R$ of recovered individual, defined by
$u^R_t(x) = 1 - u^S_t(x) - \int_0^\infty u^I_t(x, a)\, \Rd a$  and solution to $\partial_t u^R_t(x) = \int_0^\infty h(x, a)\, u^I_t(x,a) \Rd a$.
\\
While the uniqueness of the system in  Proposition~\ref{prop_PDE_eq} entails the one of the above system (by defining 
$\bar\mu^S_t(\Rd x) =  u^I_t(x, a)\, \bar \mu_X(\Rd x)$
and $\bar\mu^I_t(\Rd x, \Rd a)
= u^I_t(x, a)\, \bar \mu_X(\Rd x)\, \Rd a$), 
the regularity of such solutions and in what stronger sense the system \eqref{eq_lim_PDE_syst} is satisfied is left open.

\begin{proof}
	By taking derivative with respect to $t$ in the expression of $ \langle \bar\mu^{I}_t, \psi\rangle$ in \eqref{eq_def_bmuI},   we obtain
	\begin{align*}
		\frac{\Rd}{\Rd t} \langle \bar\mu^{I}_t, \psi\rangle  &= \int_0^\infty\hcm{-0.3} \int_{\bX} \partial_a \psi(x, a+t)  \frac{F^c_x(a+t)}{F^c_x(a)} \bar\mu_0^I(\Rd x,\Rd a) \\
		& \quad - \int_0^\infty\hcm{-0.3} \int_{\bX} \psi(x, a+t)  \frac{f_x(a+t)}{F^c_x(a)} \bar\mu_0^I(\Rd x,\Rd a) \\ 
		& \quad + \int_{\bX} \psi(x,0) \overline{\mfF}(t,x) \bar\mu^S_t(\Rd x)  \\
		& \quad - \int_0^t\hcm{-0.2} \int_{\bX} \psi(x,t-s)  f_x(t-s)\overline{\mfF}(s,x) \bar\mu^S_s(\Rd x)\Rd s  \\
		& \quad + \int_0^t\hcm{-0.2} \int_{\bX} \partial_a \psi(x,t-s)  F^c_x(t-s)\overline{\mfF}(s,x) \bar\mu^S_s(\Rd x)\Rd s  \,.
	\end{align*}
	
	Next,  from \eqref{eq_def_bmuI}, we observe that
	\begin{align*}
		\langle \bar\mu^{I}_t, \partial_a \psi\rangle  &=  \int_0^\infty\hcm{-0.3} \int_{\bX} \partial_a \psi(x, a+t)  \frac{F^c_x(a+t)}{F^c_x(a)} \bar\mu_0^I(\Rd x,\Rd a) \\
		& \quad + \int_0^t\hcm{-0.2} \int_{\bX} \partial_a \psi(x,t-s) F^c_x(t-s) \overline{\mfF}(s,x) \bar\mu^S_s(\Rd x)\Rd s\,,
	\end{align*}
	and
	\begin{align*}
		\langle \bar\mu^{I}_t, h\psi\rangle  
		&= \int_0^\infty\hcm{-0.3} \int_{\bX} \psi(x, a+t)  \frac{f_x(a+t)}{F^c_x(a)} \bar\mu_0^I(\Rd x,\Rd a) \\
		& \quad + \int_0^t\hcm{-0.2} \int_{\bX} \psi(x,t-s)  f_x(t-s)\overline{\mfF}(s,x) \bar\mu^S_s(\Rd x)\Rd s \,. 
	\end{align*}
	Hence, the last three identities lead to the expression in \eqref{eqn-barmuI-PDE-1}.	
	\medskip
	
Next, \eqref{eq_def_bmuI}
entails the following formula 
\begin{equation}\label{eq_prop_muI}
	\begin{split}
		\langle \bar\mu^I_t,\psi\rangle 
		&=\int_t^\infty\int_\bX \psi(x, a)\frac{F^c_x(a)}{F^c_x(a-t)}\bar\mu^I_0(\Rd x,\Rd a-t)\\
		&\quad+\int_0^t\int_\bX \psi(x, a)F^c_x(a)\overline\mfF(t-a,x)\bar{\mu}^S_{t-a}(\Rd x)\Rd a\,. 
	\end{split}
\end{equation}
We see from this formula that the restriction of the measure $\bar\mu^I_t$ to the set $\bX\times[0,t)$ is absolutely continuous w.r.t. the measure 
$\bar{\mu}^S_{t-a}(\Rd x)\Rd a$, hence the existence of the ``density'' $\tilde{\mu}^I_s(\Rd x,a)$, whose value at $a=0$ is specified above
by the boundary condition. 
\medskip

	Now integrating  \eqref{eqn-barmuI-PDE-1} over the interval $[0,t]$, we obtain 
	\begin{equation}\label{integPDE}
		\langle \bar\mu^I_t,\psi\rangle=\langle\bar\mu^I_0,\psi\rangle+ \int_0^t\langle \bar\mu^I_s,\partial_a\psi-h\psi \rangle ds+\int_0^t\int_\bX \psi(x,0)\overline\mfF(s,x)\bar{\mu}^S_s(\Rd x) \Rd s\,.
	\end{equation}
	We now choose in \eqref{integPDE} $\psi_n(x,a)=\varsigma(x)(1-na)^+$ for some $\varsigma\in C^1(\RR_+)$. Since $ \bar\mu^I_0(\bX\times\{0\})=0$,
	we deduce from \eqref{eq_prop_muI}
	that $ \bar\mu^I_t(\bX\times\{0\})=0$ holds for any $t\ge 0$. We have 
\[\langle \bar\mu^I_t,\psi_n\rangle \to0\quad \text{and }
	\int_0^t\langle \bar\mu^I_s,h\psi_n \rangle \Rd s\to0,\ \text{ as }n\to\infty. \]
	Also $\psi_n(x,0)=\varsigma(x)$, while
	\[\int_0^t\langle \bar\mu^I_s,\partial_a\psi_n \rangle \Rd s=-n\int_0^{1/n}\int_\bX \varsigma(x)\bar{\mu}^I_s(\Rd x,\Rd a)\,.\]
	We deduce from the above computations that, as $n\to\infty$,
	\[ n \int_0^t \Rd s\int_0^{1/n}\bar{\mu}^I_s(\Rd x,\Rd a)
	\Rightarrow\int_0^t\overline\mfF(s,x)\bar{\mu}^S_s(\Rd x)\Rd s,\]
	in the sense of weak convergence of measures on $\bX$. If we denote by $\tilde{\mu}^I_t(\Rd x,0)$ the limit as $n\to\infty$ of $n\int_0^{1/n}\bar{\mu}^I_s(\Rd x,\Rd a)$, we have 
	\begin{align*}
		\int_0^t\tilde{\mu}^I_s(\Rd x,0)\Rd s&=\int_0^t\overline\mfF(s,x)\bar{\mu}^S_s(\Rd x)\Rd s,\ \text{ hence also}\\
		\tilde{\mu}^I_s(\Rd x,0)&= \overline\mfF(s,x)\bar{\mu}^S_s(\Rd x)\,.
	\end{align*}
	From this, by the integration by parts formula, we obtain 
	\begin{equation} \label{eqn-IntByParts}
		\int_0^t \langle \bar\mu^I_s,\partial_a\psi \rangle \Rd s+\int_0^t\int_\bX \psi(x,0)\overline\mfF(s,x)\bar{\mu}^S_s(\Rd x)\Rd s=-\int_0^t\langle \partial_a\bar\mu^I_s,\psi\rangle \Rd s. 
	\end{equation}
	Hence, from \eqref{eqn-barmuI-PDE-1}  and  \eqref{eqn-IntByParts}, we obtain the PDE model in  \eqref{eqn-barmuI-PDE} with the boundary condition at $a=0$: $\tilde{\mu}^I_s(\Rd x,0)= \overline\mfF(s,x)\bar{\mu}^S_s(\Rd x)$.
	\medskip

	We now prove uniqueness.  
	The argument which follows is based on the characteristics of the equation,
	as in \cite[Section~3.2]{Ev10},
	yet adapted for measure-valued solutions in a way reminiscent of duality approaches as in \cite[Section 5.5]{Daw94}.
	
Let us consider in addition to $\bar \mu^I$
	any arbitrary solution $(\wtd \mu_t^I)$ to the PDE in \eqref{eqn-barmuI-PDE-1}
	that satisfies the initial condition $\wtd \mu_0^I = \bar \mu^I_0$,
	and define $\Delta \mu^I(\Rd x) = \wtd \mu^I(\Rd x) - \bar \mu^I(\Rd x)$.
	For any $t>0$,
	we define as follows the function $\Psi_t\in  C^{0, 1}(\bX\times \RR_+)$
	in terms of $\psi_0\in C^{0, 1}(\bX\times \RR_+)$
	and $\varphi_0 \in C^{0, 1}(\bX\times [0, t])$:
	\begin{equation}
		\Psi_t(x, a)
		=\begin{cases}
			\psi_0(x, a-t)\cdot \exp\Big[\int_{a-t}^a h(x, a')\Rd a'\Big]
			\qquad \text{ for any } x\in \bX, a\in (t, \infty),
			\\
			\varphi_0(x, t-a)\cdot \exp\Big[\int_{0}^a h(x, a')\Rd a'\Big]
			\qquad  \text{ for any } x\in \bX, a\in [0, t].
		\end{cases}
		\label{eq_def_Psi_t}
	\end{equation}
The interest of this definition lies in the relation 
	between the time-derivatives in $t$ and in $a$, 
	that makes the process $\langle \Delta \mu^I_t, \Psi_t\rangle$ 
stay constant as stated next in \eqref{eq_comp_DmuI_equal}.

	Given that $h$ is a bounded continuous function,
	$(\Psi_t(x,a))_{t, x, a} \in \cC^{1, 0, 1}(\RR_+\times\bX\times  \RR_+)$. 
	Let us compute the relevant partial derivatives for our concern,
	first for any $x\in \bX$ and any $a\in (t, \infty)$, 
	\begin{align}
		\label{eq_cmp_ppsi1}
			\partial_t \Psi_t(x, a)
			&=  - \partial_a \psi_0(x, a-t)  \exp\Big[\int_{a-t}^a h(x, a')\Rd a'\Big]
			+ \Psi_t(x, a)\cdot h(x, a-t),
			\\
			\partial_a \Psi_t(x, a)
			&=   \partial_a \psi_0(x, a-t)\;  \exp\Big[\int_{a-t}^a h(x, a')\Rd a'\Big]
			+ \Psi_t(x, a)\cdot \big[h(x, a) - h(x, a-t)\big].
\notag
	\end{align}
	On the other hand, for any $x\in \bX$ and any $a\in [0, t]$, 
	\begin{equation}
		\begin{split}
			\partial_t \Psi_t(x, a)
			&=   \partial_a \varphi_0(x, t-a)  \exp\Big[\int_{0}^a h(x, a')\Rd a'\Big],
			\\
			\partial_a \Psi_t(x, a)
			&=  - \partial_a \psi_0(x, t-a)  \exp\Big[\int_{0}^a h(x, a')\Rd a'\Big]
			+ \Psi_t(x, a)\cdot h(x, a).
		\end{split}
		\label{eq_cmp_ppsi2}
	\end{equation}
	Thanks to \eqref{eq_cmp_ppsi1} and \eqref{eq_cmp_ppsi2}, we obtain 
	\begin{equation*}
		\partial_t \Psi_t + \partial_a \Psi_t - h\cdot \Psi_t
		\equiv 0.
	\end{equation*}
	Since  $\wtd \mu^I, \bar \mu^I$ are solutions to the PDE in \eqref{eqn-barmuI-PDE-1} 
	and satisfy the same initial condition $\wtd \mu_0^I = \bar \mu^I_0$,
	the above identity implies the following one:
	\begin{equation}
		\langle \Delta \mu^I_t, \Psi_t\rangle
		= \langle \Delta \mu^I_0, \psi_0\rangle
		= 0,
		\label{eq_comp_DmuI_equal}
	\end{equation}
	which, for any $t>0$, holds for any $\psi_0, \varphi_0\in \cC^{0, 1}(\bX\times \RR_+)$.
	
	We next verify that 
	for any $\psi\in \cC^{0, 1}(\bX\times \RR_+)$,
	we can choose $\psi_0$ and $\varphi_0$ such that $\Psi$
	given by \eqref{eq_def_Psi_t} satisfies $\Psi_T(x, a) = \psi(x, a)$.
	Since $h(x, a) = -\partial_a \log F_x^c(a)$,
	we know that $\exp\Big[-\int_{0}^a h(x, a')\Rd a'\Big] =  F_x^c(a)>0$.
	Let $T>0$ and $\psi\in \cC^{0, 1}(\bX\times \RR_+)$.
	For $\psi$ to agree with $\Psi_T$, we define for any $x\in \bX$,
	any  $a\in \RR_+$ and any $r\in [0, T]$:
	\begin{equation*}
		\begin{split}
			\psi_0(x, a) &= \psi(x, a+T)\cdot \frac{F_x^c(a+T)}{F_x^c(a)},
			\qquad
			\varphi_0(x, r)
			=  \frac{\psi(x, T-r)}{F_x^c(T-r)}.
		\end{split}
	\end{equation*}
	We then check,
	for any $a>T$, that $\psi_0(x, a-T)\cdot \exp\Big[\int_{a-T}^a h(x, a')\Rd a'\Big]$
	agrees with $\psi(x, a)$
	and similarly for any $a\in [0, T]$ with $\varphi_0(x, T-a)\cdot \exp\Big[\int_{0}^a h(x, a')\Rd a'\Big]$ instead.
	Thanks to~\eqref{eq_comp_DmuI_equal},
	with the specific choice of $t= T$, 
	we deduce that $\langle \wtd \mu_T^I, \psi\rangle
	= \langle \bar \mu_T^I, \psi\rangle$
	holds for any $\psi\in \cC^{0, 1}(\bX\times \RR_+)$.
	Since this identity is valid for any $T$ provided $h$ is bounded continuous,
	the uniqueness of the solution to the PDE in \eqref{eqn-barmuI-PDE-1} is deduced. 
	
Finally,
	if $(\wtd \mu_t^I)$ is instead assumed to be any solution to the PDE in \eqref{eqn-barmuI-PDE},
	with the boundary condition at $a=0$
	as specified in Proposition~\ref{prop_PDE_eq}.
Then, the integration by parts formula in \eqref{eqn-IntByParts} 
holds with $\wtd \mu^I$
instead of $\bar \mu^I$,
and so \eqref{integPDE}
similarly for any $t>0$.
Therefore, $\wtd \mu^I$
is actually solution to the PDE in \eqref{eqn-barmuI-PDE-1}.
So it coincides with $\bar \mu^I$
by the preceding uniqueness result.
\end{proof}

\medskip

\section{Proof of Theorem~\ref{thm_cvg}, our main result} \label{sec:proofs}

\subsection{Convergence of $(\bar \mu^{S, N}_{\cdot}, \overline{\mfF}^N(\cdot))$}

We start by constructing an auxiliary model using the limit $\overline{\mfF}$. Recall $ D^N_i (t) $ in \eqref{eq_def_DiN}. 
Define
\begin{equation}
	 \wtd D^N_i (t) 
 =   \int_0^t \int_0^\infty \idc{\wtd D^N_i(s^-)=0} 
 \idc{u \le \overline \mfF(s, X^N_i)} Q_i(\Rd s,\Rd u)\,,
\label{eq_def_wtdDNi}
\end{equation}
and 
\begin{equation}
\langle 	\wtd \mu^{S,N}_t, \varphi\rangle   = \frac{1}{N}\sum_{i\in \cS^N(0)} \idc{\wtd D^N_i(t)=0}\varphi(X_i^N)\,.
\label{eq_def_tmuS}
\end{equation}

Considering this approximation considerably helps to justify the proximity 
with the limiting measure $\bar \mu^S$, as we can see from the next Lemma~\ref{lem_cvg_wmu_S}.
\begin{lem}\label{lem_cvg_wmu_S}
As $N\to\infty$,
the following convergence holds  in probability 
for any $t$ and any bounded continuous function $\varphi$ from $\bX$ to $\RR$:
\[
\langle	\wtd{\mu}^{S,N}_t, \varphi\rangle 
\to \langle	\bar\mu^S_t, \varphi\rangle,
\]
where the limit $ \bar\mu^S$ is the first term of the unique solution of  \eqref{eq_def_muSt1}-\eqref{eq_def_bar_mfFtx}.
\end{lem}

\begin{proof}
We start with the formula \eqref{eq_def_tmuS}.
Noting that
\[
\PR_0^N(\wtd D^N_i(t)=0)=\exp\left(-\int_0^t \overline \mfF(s, X^N_i)\Rd s\right),
\]
we deduce
\begin{equation}
	\bE_0^N\Big[\langle	\wtd{\mu}^{S,N}_t, \varphi\rangle\Big]
	= \textstyle\langle	\bar\mu^{S,N}_0, \varphi\cdot \exp\big(-\int_0^t \overline \mfF(s, .)\Rd s\big)\rangle.
	\label{eq_exp_wmu_SN}
\end{equation}
Recall from Assumption \ref{hyp-initial} that $\bar\mu^{S,N}_0$ converges in probability to $\bar \mu^S_0\le \bar \mu_X$,
and from Lemma \ref{lem_reg_bF} that the (deterministic) function 
$\varphi\cdot \exp\big(-\int_0^t \overline \mfF(s, .) \Rd s\big)$
is $\bar \mu_X$-a.e. continuous and bounded.
Thanks to the Portmanteau theorem, 
we thus deduce the convergence in probability of the expectation in \eqref{eq_exp_wmu_SN} to 
\begin{equation}\label{eq_lim_cvg_muS}
	\textstyle\langle	\bar\mu^{S}_0, \varphi\cdot \exp\big(-\int_0^t \overline \mfF(s, .)\Rd s\big)\rangle
	= \langle	\bar\mu^{S}_t, \varphi\rangle.
\end{equation}
On the other hand, we control the fluctuations through the variance,
by exploiting the independence property of $\wtd D_i^N$ between individuals $i$. 
Similarly as $\bE_0^N$ denotes the expectation conditional on $\cF_0^N$,
$\Var_0^N$ denotes the variance conditional on $\cF_0^N$,
in the sense that $\Var_0^N(Z) = \bE_0^N[Z^2] - \bE_0^N[Z]^2$ for any random variable $Z$.
We have 
\begin{equation}
	\Var_0^N(\langle	\wtd{\mu}^{S,N}_t, \varphi\rangle)
	= \frac1{N^2} \sum_{i\in \cS^N(0)} \varphi(X_i^N)^2\cdot \Var_0^N\Big(\idc{\wtd D_i^N(t) = 0}\Big)
	\le \frac{\Ninf{\varphi}^2}{N}.
\end{equation}	
Let $\eps>0$.
By choosing $N$ larger than $\eps^{-3}\cdot \Ninf{\varphi}^2$,
we deduce as a consequence of Chebyshev's inequality:
\begin{equation*}
\PR\Big(\big| \langle	\wtd{\mu}^{S,N}_t, \varphi\rangle 
- \bE_0^N\big[\langle	\wtd{\mu}^{S,N}_t, \varphi\rangle\big]
\big|\ge \eps \Big)
\le \eps.
\end{equation*}
So this sequence of centered random variables converges in probability to 0.
With the convergence in probability of the conditional expectation
	stated in \eqref{eq_lim_cvg_muS}, 
this concludes the proof of Lemma~\ref{lem_cvg_wmu_S}.
\end{proof}

In order to relate $\widetilde \mu^{S, N}$ to our original process $\bar \mu^{S, N}$,
we will study the convergence of the following quantity, notably with the forthcoming Proposition \ref{prop_cvg_DN}:
\begin{equation}\label{eq_def_mfD}
\mfD(t)
:=	\bE_0^N\Bigg[\frac1N \sum_{i\in \cS^N(0)}\sup_{r \in[0,t]}|D^N_i(r) - \wtd D^N_i(r)|\Bigg]
\,,
\end{equation}
We relate this convergence to the 
differences $\overline\mfF_i^N(s) - \overline\mfF(s, X_i^N)$
for $s\ge 0$ and $i\in \cS^N(0)$.
We deduce from \eqref{eq_def_mfFiN} and \eqref{eq_def_bar_mfFtx}
that those differences can then be decomposed into seven terms,
by exploiting the following definitions.

We first define 
\begin{equation}
	\overline \mfA_i^N(t)
	=  \sum_{j\in \cS^N(0)} \omega^N(i, j)\cdot \big[\lambda^N_j(t -\tau_j^N) 
	- \lambda^N_j(t - \wtd \tau_j^N)\big],
	\label{eq_def_cAiNt}
\end{equation}
where $\wtd \tau_j^N$ is the jump time 
of the process $(\wtd D^N_j (s))_{s>0}$ in 
\eqref{eq_def_wtdDNi},
so that $\overline \mfA_i^N(t)$ captures the discrepancies
between the jump of $(D^N_j)$ and that of $(\wtd D^N_j)$.

\begin{rem}
	Similarly to $\sD^N(t)$ in \eqref{eq_def_DN}, we can define $\wtd \sD^N(t)$ accordingly to $(\wtd \tau_j^N)$:
\begin{equation}
\wtd \sD^N(t)
= \big\{ j\in \cS^{N}_0;\, \wtd \tau_j^N \le t\big\}\,,
\label{eq_def_wDN}
\end{equation}
that is the subset of individuals infected
by time $t$,
while being affected by the mean-field infection rate.
For any $j\in \wtd \sD^N(t)$,
$\wtd A_j^N(t) = t - \wtd \tau_j^N$
can be interpreted as the corresponding infection age,
 of individual $j$  at time $t$, 
 like $A_j^N(t) = t - \tau_j^N$ for any $j\in \sD^N(t)$.
Possibly $\wtd A_j^N(t) > \eta_j^N$, thus the individual $j$ has recovered by time $t$ 
and $\lambda^N_j(t - \wtd \tau_j^N) = 0$.
 For any $j\in \cS^{N}_0\setminus \wtd \sD^N(t)$ on the other hand,
 $t - \wtd \tau_j^N<0$, which entails also $\lambda^N_j(t - \wtd \tau_j^N) = 0$.
 Similar observations hold for the value of $\lambda^N_j(t -\tau_j^N)$
 depending on whether $j\in \sD^N(t)$ or not, and if yes whether $A_j^N(t) > \eta_j^N$ holds or not.
It justifies the statement 
that $\overline \mfA_i^N(t)$ corresponds to the component  
due to the discrepancies between infection ages.
\end{rem}

We next define
\begin{multline}
	\overline \mfV_i^{N, 1}(t)
	=  \frac{1}{N}\sum_{j\in \cS^N(0)}   
	\Big[N \omega^N(i, j)\cdot\lambda^N_j(t-\wtd \tau_j^N) 
	- \bar\omega^N(X_i^N, X_j^N)\cdot 	\bE_0^N[\bar\lambda(X_j^N, t-\wtd \tau_j^N)]\Big],
	\label{eq_def_cViNt}
\end{multline}
so that $\overline \mfV_i^{N, 1}$ concerns the approximation of
the transmission rate by its average,
where the expectation in the last term averages the randomness of the infection time $\wtd \tau_j^N$
with the exterior field $\overline \mfF(., X_j^N)$ acting on individual $j$.

For any $x\in \bX$  and $t\ge 0$, we define
\begin{equation}
	\overline \mfL^{N, 1}(t, x)
	= \int_{\bX} 	\Big[\bar \omega^N(x, x') 
	- 	\bar \omega(x, x') \Big]
	 \cdot \bE[\bar\lambda(x', t-\wtd \tau_{x'})]
 \bar\mu^{S, N}_0(\Rd x')\;,
	\label{eq_def_cLiNt}
\end{equation}
where we define the random time $\wtd \tau_{x'}$ whatever $x'\in \bX$
as follows in term of some Poisson random measure $Q$ on $\RR_+^2$ 
with intensity $\Rd s \Rd u$:
\begin{equation}\label{eq_def_td_tauX}
\wtd \tau_{x'} 
= \inf\Big\{t\ge 0;\, \int_0^t \int_0^\infty 
\idc{u \le \overline \mfF(s, x')} Q(\Rd s,\Rd u) \ge 1\Big\}\,.
\end{equation}
Thus, $\overline \mfL^{N, 1}$ concerns the approximation of $\bar \omega^N$
by the kernel $\bar \omega$. 
\begin{rem}
The time $\wtd \tau_{x'}$ will only be considered through expectations taken at fixed $x'$ value,
so that we are not concerned about letting the random measure $Q$ depend on $x'$.
\end{rem}
The approximation 
of the initial condition $\bar\mu^{S, N}_0$
	by $\bar\mu^S_0$
is treated separately with the next term, defined also for any $x\in \bX$  and $t\ge 0$, 
\begin{equation}
		\overline \mfE^{N, 1}(t, x)
= \int_{\bX}	\bar \omega(x, x') \cdot \bE[\bar\lambda(x', t-\wtd \tau_{x'})]\;
			\Big[\bar \mu^{S, N}_0 - \bar\mu^S_0\Big](\Rd x').
	\label{eq_def_cEiNt}
\end{equation}
We will see in Remark~\ref{rem_cEiNt} that
$\overline \mfE^N(t, X^N_i)$
can also be related to the approximation 
of the empirical measure process $(\wtd \mu^{S, N}_s)$
by the limiting  $(\bar\mu^S_s)$.

We then define 
\begin{multline}
	\overline \mfV_i^{N, 0}(t)
	=  \frac{1}{N}\sum_{j\in \cI^N(0)}  
	\Big[N \omega^N(i, j)\lambda^N_j(A_j^N(0) + t) 
	- \bar\omega^N(X_i^N, X_j^N)
	\frac{\bar\lambda(X_j^N, A_j^N(0) + t)}{F^c_{X_j^N}(A_j^N(0))}\Big]\,,
	\label{eq_def_cViN0t}
\end{multline}
in a similar way as $\overline \mfV_i^{N,1}(t)$, now for the initially infected individuals.
We recall from Assumption~\ref{hyp-Blambda} that 
	for any $j\in \cI^N(0)$
	the value of the conditional expectation
	$\bar\lambda(X_j^N, A_j^N(0) + t)$
	is amplified
by the denominator $F^c_{X_j^N}(A_j^N(0))$
to account for the bias in $\lambda^N_j(A_j^N(0) + t)$
 due to the fact 
that individual $j$ has not recovered by time $0$.
The next terms $\overline \mfL^{N, 0}$ and $\overline \mfE^{N, 0}$ are similarly the analogs 
of respectively $\overline \mfL^{N, 1}$ and $\overline \mfE^{N, 1}$, now  for the initially infected individuals.
For any $x\in \bX$ and $t\ge 0$,
\begin{equation}
	\overline \mfL^{N,0}(t, x)
	= \int_{\bX}\hcm{-0.1}\int_0^\infty 
	\Big[\bar \omega^N(x, x') 
- 	\bar \omega(x, x') ]	
	\cdot 
	\frac{\bar\lambda(x', a' + t)}{F^{c}_{x'}(a')}
\;	\bar\mu^{I, N}_0(\Rd x', \Rd a')\;,
	\label{eq_def_cLiN0t}
\end{equation}
and finally, 
\begin{equation}
	\overline \mfE^{N, 0}(t, x)
	= \int_{\bX}\hcm{-0.1}\int_0^\infty 
	\bar \omega(x, x') \cdot	 \frac{\bar\lambda(x', a' + t)}{F^{c}_{x'}(a')}
\; \Big[\bar\mu^{I, N}_0(\Rd x', \Rd a')-	\bar\mu^I_0(\Rd x', \Rd a') \Big] \,,
	\label{eq_def_cEiN0t}
\end{equation}
so that $\overline \mfE^{N, 0}(t, X^N_i)$ accounts for the approximation of
	the initial condition $\bar\mu^{I, N}_0$ by $\bar\mu^I_0$.

\begin{lem}
		\label{lem_dec_FN}
With the above definitions, for any $N\ge 1$, $i\in \cS^N(0)$, $t\ge 0$, we have 
\begin{align*}
		\overline{\mfF}_i^N(t) - \overline{\mfF}(t, X_i^N)
 &= 	\overline \mfA_i^N(t) 
 && \text{(discrepancies between the jumps)}
\\&+ \overline \mfV_i^{N, 0}(t) + \overline \mfV_i^{N, 1}(t)
&&\text{(Law of Large Number)}
\\&+  \overline \mfL^{N, 0}(t, X^N_i) + \overline \mfL^{N, 1}(t, X^N_i) 
& &\text{(discrepancies between the kernels)}
\\&+ \overline \mfE^{N, 0}(t, X^N_i)
+ \overline \mfE^{N, 1}(t, X^N_i)
&& \text{(discrepancies between the initial conditions)}.
\end{align*} 
\end{lem}

\begin{proof}
	From \eqref{eq_def_td_tauX} it follows
	that $\PR(\wtd \tau_{x'} >s)
	= \exp\big(\int_0^s \overline{\mfF}(r,x') \Rd r\big)$,
	hence the law of $\wtd \tau_{x'}$
	has the density $\overline{\mfF}(s,x')
\,	\exp[\int_0^s \overline{\mfF}(r,x') \Rd r]$.
	This fact combined with  \eqref{eq_def_bmuS2} 
leads to 
\begin{equation}\label{eq_id_ENi}
\int_0^t\hcm{-0.2} \int_{\bX} \bar \omega(x,x') \bar\lambda(x', t-s) \overline{\mfF}(s,x') \bar\mu^S_s(\Rd x')\Rd s 
= 	\int_\bX \bar \omega(x, x') \bE[\bar\lambda(x', t-\wtd \tau_{x'})] \bar \mu^S_0(\Rd x')\,.
\end{equation}
Plugging this identity in 	\eqref{eq_def_bar_mfFtx},
we deduce the decomposition
	$\overline{\mfF}(t, x) = \overline{\mfF}^0(t, x) + \overline{\mfF}^1(t, x)$,
where
\begin{align}
	&\overline{\mfF}^0(t, x)
	=\int_{\bX}\hcm{-0.1}\int_0^\infty   \bar \omega(x,x') 
\cdot	\frac{\bar\lambda(x', a'+t)}{F^c_{x'}(a')}   \; \bar\mu_0^I(\Rd x', \Rd a')\,, 
		\label{eq_prop_mfF0}
	\\&\overline{\mfF}^1(t, x)
	= \int_\bX \bar \omega(x, x') \cdot \bE[\bar\lambda(x', t-\wtd \tau_{x'})] 
\;	\bar \mu^S_0(\Rd x')\,.
	\label{eq_prop_mfF1}
\end{align}
We aim at a similar decomposition for $\overline{\mfF}_i^N(t)$.
Recall \eqref{eq_def_DN}
where $\sD^N(t)$ has been defined as the subset in $\cS^{N}_0$ 
of individuals that have been infected by the disease  by time $t$.
If $j\in \cI^N(t)\cap \cS^{N}_0$, 
then $j \in \sD^N(t)$ and
thus $A_j^N(t) = t-\tau_j^N$.
If, on the other hand, $j\in \cS^{N}_0\setminus \sD^N(t)$, then $A_j^N(t) = 0$,
and thus $\lambda^N_j(A_j^N(t)) =0$.

Recalling \eqref{eq_def_mfFiN2},
we thus get the decomposition 
$\overline{\mfF}_i^N(t) = \overline{\mfF}_i^{N, 0}(t) + \overline{\mfF}_i^{N, 1}(t) $,
where
\begin{align}
		&\overline{\mfF}_i^{N, 0}(t)
	=\sum_{j \in \mathcal{I}^N(0)}  \omega^N(i, j)\cdot\lambda^N_j(A_j^N(0) + t) 
		\label{eq_prop_mfFN0}
	\\&\overline{\mfF}_i^{N, 1}(t)
	= \sum_{j\in  \cS^N(0)} \omega^N(i, j)\cdot \lambda^N_j(t-\tau_j^N)\,.
		\label{eq_prop_mfFN1}
\end{align}

We first show that $\overline{\mfF}_i^{N, 0}(t) - \overline{\mfF}^0(t, X_i^N)
= \overline \mfV_i^{N, 0}(t) + \overline \mfL^{N, 0}(t, X^N_i) + \overline \mfE^{N, 0}(t, X^N_i)$,
and next that $\overline{\mfF}_i^{N, 1}(t) - \overline{\mfF}^1(t, X_i^N)
= \overline \mfA_i^N(t) + \overline \mfV_i^{N, 1}(t) + \overline \mfL^{N, 1}(t, X^N_i) + \overline \mfE^{N, 1}(t, X^N_i)$.

By combining \eqref{eq_def_cViN0t} with \eqref{eq_prop_mfFN0},
then with the definition of $\bar\mu^{I, N}_0$  in \eqref{eq_def_muINt1}, we have
\begin{equation*}
\begin{split}
	\overline{\mfF}_i^{N, 0}(t)
- \overline \mfV_i^{N, 0}(t)
&= \frac{1}{N}\sum_{j\in \cI^N(0)}  
\bar\omega^N(X_i^N, X_j^N)
\frac{\bar\lambda(X_j^N, A_j^N(0) + t)}{F^c_{X_j^N}(A_j^N(0))}
\\&
=\int_{\bX}\hcm{-0.1}\int_0^\infty   
\bar \omega^N(X_i^N, x') \cdot \frac{\bar\lambda(x', a' + t)}{F^{c}_{x'}(a')}
\;\bar\mu^{I, N}_0 (\Rd x', \Rd a')\,.
\end{split}
\end{equation*}
Plugging \eqref{eq_def_cLiN0t} into this expression,
then exploiting \eqref{eq_def_cEiN0t} and \eqref{eq_prop_mfF0}, we obtain
\begin{equation}
\begin{split}
	\overline{\mfF}_i^{N, 0}(t)
- \overline \mfV_i^{N, 0}(t)
- \overline \mfL^{N, 0}(t, X^N_i)
&= \int_{\bX}\hcm{-0.1}\int_0^\infty \bar \omega(X^N_i, x') \cdot\frac{\bar\lambda(x', a' + t)}{F^{c}_{x'}(a')} 
  \bar\mu^{I, N}_0(\Rd x', \Rd a')
\\&= \overline \mfE^{N, 0}(t, X^N_i)
+  \overline{\mfF}^0(t, X_i^N)\,,
\end{split}
\label{eq_dec_FN0}
\end{equation}
which concludes our first claim. For the second claim, there is a first additional step where $A_i^N(t)$ is related to $\wtd A_i^N(t) = t- \wtd \tau_i^N$, through $\overline \mfA_i^N(t)$, before we can exploit the same arguments. We first combine \eqref{eq_def_cAiNt} with \eqref{eq_prop_mfFN1}, and hence, 
\begin{equation*}
	\overline{\mfF}_i^{N, 1}(t)
- \overline \mfA_i^N(t)
=  \sum_{j\in \cS^N(0)} \omega^N(i, j) \lambda^N_j(t-\wtd \tau_j^N)\,.
\end{equation*}
Plugging \eqref{eq_def_cViNt} into this expression
and exploiting the definitions of $\bar \mu^{S, N}$ in \eqref{eq_def_muSNt1}
and  of $\wtd \tau_{x'}$ in \eqref{eq_def_td_tauX}, we deduce that 
\begin{equation}
	\begin{split}
		\overline{\mfF}_i^{N, 1}(t)
	- \overline \mfA_i^N(t)	- \overline \mfV_i^{N, 1}(t)
		&= \frac{1}{N}\sum_{j\in \cS^N(0)}   
	\bar\omega^N(X_i^N, X_j^N)\cdot 	\bE_0^N[\bar\lambda(X_j^N, t-\wtd \tau_j^N)]
	\\&= \int_\bX 
	\bar \omega^N(X_i^N, x') \bE[ \bar\lambda(x', t-\wtd \tau_{x'})]
	\bar \mu^{S, N}_0(\Rd x')\,.
	\end{split}
\label{eq_prop_int2}
\end{equation}
Next plugging \eqref{eq_def_cLiNt} into this expression,
then exploiting \eqref{eq_def_cEiNt} and \eqref{eq_prop_mfF1}, we get 
\begin{equation}
\begin{split}
\overline{\mfF}_i^{N, 1}(t)
- \overline \mfA_i^N(t)	- \overline \mfV_i^{N, 1}(t) - \overline \mfL^{N, 1}(t, X^N_i)
&= \int_\bX 
\bar \omega(X_i^N, x') \bE[ \bar\lambda(x', t-\wtd \tau_{x'})]
\bar \mu^{S, N}_0(\Rd x')
\\&= \overline \mfE^{N, 1}(t, X^N_i) + \overline{\mfF}^1(t, X_i^N).
\end{split}
\label{eq_dec_FN1}
\end{equation}
Since $\overline{\mfF}(t, x) = \overline{\mfF}^0(t, x) + \overline{\mfF}^1(t, x)$
and $\overline{\mfF}_i^N(t) = \overline{\mfF}_i^{N, 0}(t) + \overline{\mfF}_i^{N, 1}(t)$,
combining \eqref{eq_dec_FN0} and \eqref{eq_dec_FN1}
concludes the proof of Lemma~\ref{lem_dec_FN}.
\end{proof}

\begin{rem}\label{rem_cEiNt}
In the expression of $\overline \mfE^{N, 1}(t, x)$ in \eqref{eq_def_cEiNt},
we decided to relate as directly as possible 
to the difference between $\bar \mu^{S, N}_0$ and $\bar\mu^{S}_0$.
That being said, this term has the following alternative interpretation 
in terms of the difference between the processes $\wtd \mu^{S, N}_s$ and $\bar\mu^{S}_s$:
	\begin{equation}\label{eq_def_alt_EN}
	\overline \mfE^{N, 1}(t, x)
	=   \bE_0^N\Bigg[\int_0^t  
		\int_{\bX}		\bar \omega(x, x') \bar\lambda(x', t-s)
		\cdot  \overline \mfF(s, x')\Big[ \wtd \mu^{S, N}_s - \bar\mu^S_s\Big](\Rd x')
	\;	\Rd s\, \Bigg]\,. 
\end{equation}
\end{rem}

\begin{proof}[Proof of \eqref{eq_def_alt_EN}]
	Recall the identity \eqref{eq_id_ENi}. 
	Similarly, we express $\wtd \tau_j^N$ through the Poisson random measure $Q_j$:
	\begin{align}\label{eq_prop_int}
		&\frac{1}{N}\sum_{j\in \cS^N(0)} \bar \omega(x, X_j^N)
		\bE_0^N[\bar\lambda(X_j^N, t-\wtd \tau_j^N)]
		\\&= \bE_0^N\Bigg[\frac{1}{N}\sum_{j\in \cS^N(0)} 
		\int_0^t\int_0^\infty \bar \omega(x, X_j^N) \bar\lambda(X_j^N, t-s)
		\idc{\wtd D_j^N(s^-)= 0} \idc{u\le \overline \mfF(s^-, X_j^N)} Q_j(\Rd s, \Rd u)\Bigg].
		\notag
	\end{align}
	In the expression in the second line, 
	we can replace $Q_j$ by its intensity
	since the integrant is predictable with respect to its filtration $(\wtd \cF^N_t)$,
	then express the sum over $j\in \cS^N(0)$ in terms of the 
	empirical measure $\wtd \mu^{S, N}_s$, recall \eqref{eq_def_tmuS},
	which leads to 
	\begin{equation}\label{eq_id_wmusSN}
		\bE_0^N\Big[\int_0^t
		\int_\bX 
		\bar	\omega(x, x') \bar\lambda(x', t-s)
		\overline \mfF(s, x') \wtd \mu^{S, N}_s(\Rd x') \Rd s\Big].
	\end{equation}
	Recalling \eqref{eq_prop_int2} in addition to \eqref{eq_id_ENi}, 
	\eqref{eq_prop_int} and \eqref{eq_id_wmusSN},
	we deduce identity \eqref{eq_def_alt_EN}.
\end{proof}

In the following, we treat separately the various terms distinguished in Lemma \ref{lem_dec_FN},
first $\overline \mfA_i^N$, see Lemma~\ref{lem_AN_bound},
second $\overline \mfV^{N, 1}$ and $\overline \mfV^{N, 0}$, see Lemma~\ref{lem_VN_bound},
third  $\overline \mfL^{N, 1}$ and $\overline \mfL^{N, 0}$, see Lemma~\ref{lem_LN_bound},
and finally  $\overline \mfE^{N, 1}$ and $\overline \mfE^{N, 0}$, see Lemma~\ref{lem_EN_bound}.
We start with the processes $(\overline \mfA_i^N)$
defined in~\eqref{eq_def_cAiNt}.
\begin{lem}
\label{lem_AN_bound}
Under Assumptions \ref{hyp-lambda}, \ref{hyp-initial}, \ref{hyp-w} and \ref{hyp-Var}, 
there exist a constant $C>0$
such that 
\begin{equation*}
\bE_0^N\Bigg[ \frac1N	\sum_{i\in \cS^N(0)}	\Big|\overline \mfA_i^N(t)\Big|\Bigg]
\le C\cdot \mfD(t) 
+ \overline \mfV^N\,
\end{equation*}
holds  a.s. for any $t\ge 0$,
where $\mfD(t) $ is defined in \eqref{eq_def_mfD}, and  $\overline\mfV^N$ is given by
\begin{equation} \label{eq_def_mfVN}
	\overline\mfV^N := \lambda^*\cdot \sqrt{\Upsilon^N
		+ \omega^*\cdot \bar \gamma^N}
	\,.
\end{equation}

\end{lem}

We recall the defining property of $\omega^*$ given in Assumption \ref{hyp_ws}. 

\begin{proof}
Recalling \eqref{eq_def_cAiNt}
and exploiting  the upper-bound $\lambda^* $
of the functions $(\lambda_j^N)$, we get 
\begin{equation}\label{eq_prop_DiN}
	\sum_{i\in \cS^N(0)}	\Big|\overline \mfA_i^N(t)\Big|
	\le  \lambda^* 
	\cdot  \sum_{j\in \cS^N(0)} \sup_{r\le t}|D^N_j(r) - \wtd D^N_j(r)|
	\cdot \sum_{i\in \cS^N(0)} \omega^N(i, j).
\end{equation}
We observe for any $j\in \cS^N(0)$, the following identity 
by virtue of \eqref{eq_prop_omN}:
\begin{equation*}
	\bE_0^N\Big[ \sum_{i\in \cS^N(0)} \omega^N(i, j)\Big]
	= \int_{\bX} \bar\omega^N(x, X_j^N) \bar\mu_0^{S, N}(\Rd x).
\end{equation*}
Since $\bar \omega^N$ is upper-bounded by $\omega^*$
and $\bar\mu_0^{S, N}(\bX)\le 1$, we obtain the upper bound
\begin{equation}\label{eq_prop_mfA}
\frac1N	\sum_{i\in \cS^N(0)} \Big|\overline \mfA_i^N(t)\Big|
	\le \frac{\lambda^*}{N} \sum_{j\in \cS^N(0)} \sup_{r\le t}|D^N_j(r) - \wtd D^N_j(r)|
	\cdot(\wtd S_j^N + \omega^*)\,,
\end{equation}
where 
\begin{equation}\label{eq_def_wSjN}
	\wtd S_j^N  := \sum_{i\in \cS^N(0)} \omega^N(i, j) 
	- \bE_0^N\Big[ \sum_{i\in \cS^N(0)} \omega^N(i, j)\Big]\,.
\end{equation}
Since $\sup_{r\le t}|D^N_j(r) - \wtd D^N_j(r)|$ is upper-bounded by 1
for any $j\in  \cS^N(0)$,
\eqref{eq_prop_mfA} entails 
\begin{equation*}
	\bE_0^N\Bigg[ \frac1N	\sum_{i\in \cS^N(0)}	\Big|\overline \mfA_i^N(t)\Big|\Bigg]
	\le C\cdot \mfD(t) 
	+ \overline \mfU^N\,,
\end{equation*}
where $C = \lambda^*\cdot \omega^*$ and 
\begin{equation} \label{eq_def_mfUN}
	\overline\mfU^N := \frac{\lambda^*}{N} \bE_0^N\Bigg[\sum_{j\in \cS^N(0)} \big|\wtd S_j^N\big|\Bigg]\,.
\end{equation}
So Lemma~\ref{lem_AN_bound} will be concluded by showing that $\overline\mfU^N\le 	\overline\mfV^N$.
Thanks to the Cauchy-Schwartz inequality
with respect to the product measure,
with the fact that the cardinality of
$\cS^N(0)$ is less than $N$: 
\begin{equation}\label{eq_prop_CS_wSjN}
\begin{split}
	\bE_0^N\Bigg[ \frac1{N} \sum_{j\in \cS^N(0)} \Big\vert\wtd S_j^N \Big\vert\Bigg]\,
\le \sqrt{ \frac1{N} \sum_{j\in \cS^N(0)} \bE_0^N\Bigg[ \Big(\wtd S_j^N \Big)^2\Bigg]}\,.
\end{split}
\end{equation}
For any $j\in \cS^N(0)$ and conditionally on $\cF^N_0$,
$\wtd S_j^N$ is the sum of independent centered variables,
which leads to the following identity:
\begin{equation*}
\bE_0^N\Bigg[ \Big(\wtd S_j^N \Big)^2\Bigg]
	= \sum_{i\in \cS^N(0)} 
\Var_0^N\Big[\omega^N(i, j)\Big]\,.
\end{equation*}
We recall that $\Var_0^N(Z) = \bE_0^N(Z^2) - \bE_0^N(Z)^2$ by definition 
	of this variance conditional on $\cF_0^N$ for any random variable $Z$. 
As we will need later the conditional second moment of $\omega^N(i, j)$
instead of its variance, 
we rather consider it as the upper-bound, 
then exploit \eqref{eq_prop_ups} and \eqref{eq_def_kN}
to deduce the following inequality:
\begin{multline}
\bE_0^N\Bigg[ \Big(\wtd S_j^N \Big)^2\Bigg]
\le 	\sum_{i\in \cS^N(0)} \bE_0^N\Big[\omega^N(i, j)^2\Big]
\\
\quad= \sum_{i\in \cS^N(0)} \bE_0^N\big[\upsilon^N(i, j); (i, j)\in \cE^N\big] 
 + \sum_{i\in \cS^N(0)} \kappa^N(X_i^N, X_j^N) \cdot \gamma^N(X_i^N, X_j^N)^2\,.
\end{multline}
Remark that from \eqref{eq_prop_omN},
for any $i, j\in \cS^N(0)$,
$\kappa^N(X_i^N, X_j^N)\cdot \gamma^N(X_i^N, X_j^N)
=  \bar\omega^N(X_i^N, X_j^N)/N\le \omega^*/N$. 
Recalling the definitions given in Assumption \ref{hyp-Var},
we thus deduce 
\begin{equation}\label{eq_prop_UAn}
\frac1{N} \sum_{j\in \cS^N(0)} \bE_0^N\Bigg[ \Big(\wtd S_j^N \Big)^2\Bigg]
\le \Upsilon^N
+ \omega^*\cdot \bar \gamma^N\,.
\end{equation}
Recalling \eqref{eq_def_mfUN} and \eqref{eq_prop_CS_wSjN},
this entails $\overline\mfU^N\le 	\overline\mfV^N$ a.s.
and concludes the proof of Lemma~\ref{lem_AN_bound}.
\end{proof}

For the upper-bound of $\overline \mfV_i^{N, 1}$ and $\overline \mfV_i^{N, 0}$
as defined in respectively \eqref{eq_def_cViNt} and \eqref{eq_def_cViN0t},
the proof of the next lemma follows similar principles as in the previous one.

\begin{lem}
	\label{lem_VN_bound}
The following upper-bound holds for any $t>0$
with the sequence 	$(\bar  \mfV^N)_{N\ge 1}$ defined in \eqref{eq_def_mfVN}:
	\begin{equation*}
\bE_0^N\Big[ \frac1N\sum_{i\in \cS^N(0)}	\big|\overline \mfV_i^{N, 1}(t)\big|  \Big]
\vee \bE_0^N\Big[ \frac1N\sum_{i\in \cS^N(0)}	\big|\overline \mfV_i^{N, 0}(t)\big| \Big]
\le \overline\mfV^N\,.
	\end{equation*}
\end{lem}

\begin{proof}
	We treat this component in the same way as $\mfU^N$, 
starting with  the Cauchy-Schwartz inequality:
\begin{equation*}
	\begin{split}
		\bE_0^N\Big[ \frac1N\sum_{i\in \cS^N(0)}	\big|\overline \mfV_i^{N, 1}(t)\big| \Big]
\le \sqrt{\frac1N \sum_{i\in \cS^N(0)} \bE_0^N\Big[\big(\overline \mfV_i^{N, 1}(t)\big)^2  \Big]},
	\end{split}
\label{eq_CS_VN}
\end{equation*}
Recalling \eqref{eq_def_cViNt},
we note that the random variables $\overline \mfV_i^{N, 1}(t)$ are centered
conditionally on $\cF^N_0$, 
so that the term under the square root is actually a variance.
$\overline \mfV_i^{N, 1}(t)$
is a sum of r.v.'s which are orthogonal in $L^2(\Omega, \PR_0^N)$.
Thus, for any $i\in \cS^N(0)$,
\begin{equation*}
	\bE_0^N\Big[\big(\overline \mfV_i^{N, 1}(t)\big)^2 \Big]
	= \sum_{j\in \cS^N(0)}   
	\Var_0^N\big[\omega^N(i, j)\cdot\lambda^N_j(t-\wtd \tau_j^N) \big].
\end{equation*}
For any $i, j \in \cS^N(0)$, 
since $\omega^N(i, j)$ and $\lambda^N_j(t-\wtd \tau_j^N)$ are independent conditionally on $\cF^N_0$
and since the functions $(\lambda_j^N)$ are uniformly upper-bounded by $\lambda^* $:
\begin{equation}\label{eq_prop_omLam}
	\Var_0^N\big[\omega^N(i, j)\cdot\lambda^N_j(t-\wtd \tau_j^N)\big]
	\le (\lambda^* )^2\cdot \bE_0^N\big[\omega^N(i, j)^2 \big].
\end{equation}
The argument for the following inequality
is then the same as for \eqref{eq_prop_UAn}:
\begin{equation}
\frac1N \sum_{i\in \cS^N(0)} \bE_0^N\Big[\big(\overline \mfV_i^{N, 1}(t)\big)^2 \Big]
\le (\lambda^*)^2\cdot \Big[\Upsilon^N
+ \omega^*\cdot \bar \gamma^N\Big]\,.
\end{equation}
Concerning the sequence $(\overline \mfV_i^{N, 0}(t))$, 
we first recall the following identity
for any $j\in \cI^N(0)$ as part of Assumption \ref{hyp-Blambda}:
\begin{equation}
\bE_0^N\Big[\lambda_j^N(A_j^N(0) + t)\Big]
= \frac{\bar\lambda(X_j^N, A_j^N(0) + t)}{F^c_{X_j^N}(A_j^N(0))}.
\end{equation}
Since $\omega^N(i, j)$ and $\lambda_j^N(A_j^N(0) + t)$ are independent conditionally 
on $\cF^N_0$, we deduce that $\overline \mfV_i^{N, 0}(t)$ is also conditionally centered,
whatever $i\in \cS^N(0)$ and $t\ge 0$.
We can then exploit the same argument for $\overline \mfV_i^{N, 0}(t)$ as for $\overline \mfV^{N, 1}_i(t)$,
thanks to the Cauchy-Schwartz inequality and replacing \eqref{eq_prop_omLam} by
\begin{equation*}
	\Var_0^N\big[\omega^N(i, j)\cdot\lambda^N_j(A_j^N(0) + t) \big]
	\le (\lambda^* )^2\cdot \bE_0^N\big[\omega^N(i, j)^2 \big].
\end{equation*}
Lemma \ref{lem_VN_bound} is therefore concluded with 
$\overline \mfV^{N}$ in \eqref{eq_def_mfVN}.
\end{proof}

$\overline \mfL^{N, 1}$ and $\overline \mfL^{N, 0}$
as defined respectively in \eqref{eq_def_cLiNt} and \eqref{eq_def_cLiN0t} are quite directly 
upper-bounded.
\begin{lem}
	\label{lem_LN_bound}
By virtue of Assumptions~\ref{hyp-eta}, \ref{hyp-Blambda} and  \ref{hyp-initial}:
\begin{equation*}
\Big|\overline \mfL^{N, 1}(t, x)\Big|+ \Big|\overline \mfL^{N, 0}(t, x)\Big|
\le \lambda^*\cdot \Ninf{\bar \omega^N	- \bar \omega}\,
\end{equation*}
holds for any $t>0$ and $x\in \bX$.
\end{lem}

\begin{proof}
We first upper-bound $|\bar \omega^N(x, x') - \bar \omega(x, x')|$
by $\Ninf{\bar \omega^N 	- \bar \omega}$.
By virtue of Assumption~\ref{hyp-Blambda},
we exploit $\lambda^*$
as the upper-bound 
for any $x'\in \bX$ and any $a', t\in \RR_+$
of $\bE[\bar\lambda(x', t-\wtd \tau_{x'})]$
in \eqref{eq_def_cLiNt}
and of $\bar\lambda(x', a'+t)/F^{c}_{x'}(a'+t)$.
Since $F_{x'}^c$ is non-increasing 
for any such $x'$ by virtue of Assumption~\ref{hyp-eta},
the latter upper-bound entails the one 
of $\bar\lambda(x', a'+t)/F^{c}_{x'}(a')$
in \eqref{eq_def_cLiN0t}.
Finally, by virtue of the considered scaling of $\bar\mu^{S, N}(\Rd x')$
and of $\bar\mu^{I, N}(\Rd x', \Rd a')$,
recall \eqref{eq_def_muSNt}-\eqref{eq_def_muINt},
their added masses is upper-bounded by 1,
which concludes the proof of Lemma~\ref{lem_LN_bound}.
\end{proof}

In order to finally deal with the upper-bound 
of  both $\overline \mfE^{N, 1}$ and $\overline \mfE^{N, 0}$,
defined respectively in \eqref{eq_def_cEiNt} and \eqref{eq_def_cEiN0t},
we exploit the following proposition which will be proved in Appendix~\ref{sec:appendix}.
\begin{prop}
\label{prop_cvg_mu0}
Let 
$\cY$ be a Polish space.
Let $\mu \in \cM(\bX)$, $\nu\in \cM(\cY)$
  and  $k:\bX\times \cY \mapsto \RR$ 
  be a bounded measurable function
  that is continuous $\mu \otimes \nu$ almost everywhere.
Let in addition
$(\mu^N)_{N\ge 1}$ and $(\nu^N)_{N\ge 1}$
be two sequences of possibly random  measures in  $\cM(\bX)$ and $\cM(\cY)$, respectively, 
that converge in probability respectively to $\mu$ and $\nu$, for the topology of weak convergence.
Then, the following quantity converges to zero in probability as $N$ tends to infinity:
\begin{equation*}
	\int_\bX \Big|\int_\cY k(x, y) [\nu^N - \nu](dy)\Big| \mu^N(\Rd x).
\end{equation*}
\end{prop}

We are now ready for the estimate of  $\overline \mfE^{N, 1}$ and $\overline \mfE^{N, 0}$
provided in the next lemma.
\begin{lem}
	\label{lem_EN_bound}
For any $t>0$, 
the following random variable converges to 0 in probability as $N$ tends to infinity:
\begin{equation*}
\overline\mfE^N(t)
:=\langle \bar \mu^{S, N}_0, |\overline \mfE^{N, 1}(t, .)| + |\overline \mfE^{N, 0}(t, .)|\rangle\,.
\end{equation*}
\end{lem}

\begin{proof}
		Recalling \eqref{eq_def_cEiN0t},
	the fact that 
	$\langle \bar \mu^{S, N}_0, |\overline \mfE^{N, 0}(t, .)|\rangle$
	converges to zero is a direct consequence of Proposition \ref{prop_cvg_mu0}
	with $\cY = \bX\times \RR_+$ and 
	\begin{equation*}
		\begin{split}
			\mu^N  := \bar \mu^{S, N}_0,\quad  
			\mu :=\bar \mu^{S}_0, \qquad	
			&k_t^0(x, x', a') = \begin{cases}
			&\bar\omega(x, x')\cdot \frac{\bar\lambda(x', a' + t)}{F_{x'}^{c}(a')} \quad
			\text{ if } F_{x'}^{c}(a')>0,
			\\& 0 \hcm{3.5} \text{otherwise,}
			\end{cases}
			\\	\nu^{N, 0}(\Rd x', \Rd a'):= \bar \mu^{I, N}_0(\Rd x', \Rd a'), \quad 
			&\nu^0(\Rd x', \Rd a') :=  \bar\mu^I_0(\Rd x',\Rd a') \,.
		\end{split}
	\end{equation*}
Assumption \ref{hyp-initial} 
states that $\mu^N$ and $\nu^{N, 0}$ converge weakly to respectively $\mu$ and $\nu^0$.
We recall the upper-bound of $\bar{\lambda}(x',a'+t)/F^c_{x'}(a')$
by $\lambda^*$ as a consequence of Assumptions~\ref{hyp-eta} and \ref{hyp-Blambda}.
We recall also the uniform upper-bound $\omega^*$ of $\bar \omega^N$,
as defined in Assumption~\ref{hyp_ws}.
$\omega^*\cdot \lambda^*$ is thus  a global upper-bound of $k_t^0$.

By virtue of Assumption~\ref{hyp-w} and since $t>0$,
the following property holds
for $(\mu\otimes \nu^0)$ almost every triplet $(x, x', a')$ such that  $F^c_{x'}(a') = 0$:
$F^c_{\check x'}(\check a'+t) = 0$ holds for any $(\check x, \check x', \check a')$ in a small enough neighborhood of $(x, x', a')$,
thus also $\bar\lambda(\check x', \check a' + t) = 0$ 
and $k(\check x, \check x', \check a') = 0$. 

The following property holds on the other hand, 
for $(\mu\otimes \nu^0)$ almost every triplet $(x, x', a')$ such that  $F^c_{x'}(a') > 0$:
$F^c_{\check x'}(\check a') > 0$ holds for any $(\check x, \check x', \check a')$ in a small enough neighborhood of $(x, x', a')$.
Still by virtue of Assumption~\ref{hyp-w}
on $\bar \lambda$ and $\bar \mu_0^I$,
the function $:(x', a')\mapsto \bar \lambda(x', a'+t)$ is $\nu^0$-a.e. continuous.
We thus conclude that $k_t^0$ is $(\mu\otimes \nu^0)$-a.e. continuous. 

Proposition \ref{prop_cvg_mu0} does therefore entail that $\langle \bar \mu^{S, N}_0, |\overline \mfE^{N, 0}(t, .)|\rangle$
tends to zero.

\medskip

Recalling \eqref{eq_def_cEiNt},
the fact that $\langle \bar \mu^{S, N}_0, |\overline \mfE^{N, 1}(t, .)|\rangle$
	converges to zero is a consequence as well  of Proposition \ref{prop_cvg_mu0}	
	with this time $\cY = \bX$  and 
	\begin{equation*}
	\mu^N = \nu^{N, 1} := \bar \mu^{S, N}_0,\quad  \mu = \nu^1 = \bar \mu^{S}_0, \quad
k^1_t(x, x') = \bar\omega(x, x')\cdot \bE[\bar\lambda(x', t-\wtd \tau_{x'})]\,.
	\end{equation*}
	The fact that $\mu^N$ and $\nu^{N, 1}$ converge weakly in probability to $\mu= \nu^1$
follows from Assumption~\ref{hyp-initial}.
	$k_t^1$ is bounded under Assumptions~\ref{hyp-lambda} and \ref{hyp-w}.
	To check that $k^1_t$ is $(\mu\otimes \nu^1)$ a.e. continuous,
	 we exploit the following alternative expression:
	\begin{equation*}
k_t^1(x, x') = \bar\omega(x, x') \int_0^t \bar\lambda(x', t-s)\cdot \overline \mfF(s, x')\cdot \exp\Big[-\int_0^s\overline \mfF(r, x')\Rd r\Big]\,\Rd s\,,
	\end{equation*}
	derived with the same argument as for	\eqref{eq_def_alt_EN}.
	Thanks to Assumptions~\ref{hyp-lambda} and \ref{hyp-w} and Lemma \ref{lem_reg_bF},
	we can conclude that $k^1$ is $(\mu\otimes \nu^1)$ a.e. continuous,
	as stated in Lemma~\ref{lem_ae_cont} in Appendix~\ref{sec:appendix}.
	
Proposition \ref{prop_cvg_mu0} does therefore entail that $\langle \bar \mu^{S, N}_0, |\overline \mfE^{N, 1}(t, .)|\rangle$
		tends to zero.
\end{proof}

With Lemmas \ref{lem_AN_bound}, \ref{lem_VN_bound}, \ref{lem_LN_bound} and \ref{lem_EN_bound}, 
we are ready to prove the following comparison result with the original model.
\begin{prop}
	\label{prop_cvg_DN}
As $N\to \infty$, $\mfD(t)$ defined in \eqref{eq_def_mfD} converges in probability to 0 
locally uniformly in $t$.
\end{prop} 

\begin{proof}
First note that it suffices to prove the convergence for any fixed $t$. 
The locally uniform convergence then follows from Lemma~\ref{lem_loc_unif_cvg} in Appendix~\ref{sec:appendix}
since $:t\mapsto \mfD(t)$
is a.s. non-decreasing for any $N$.
For any $i$ and $t$, 
\begin{equation*}
\sup_{r\le t}|D^N_i(r) - \wtd D^N_i(r)|
\le \int_0^t \int_{0}^\infty
\idc{u\in \cB_i^N(s)} Q_i(\Rd s, \Rd u),
\end{equation*}
where the interval $\cB_i^N(s)$ is defined as follows:
\begin{equation*}
\cB_i^N(s) = \big[\overline\mfF_i^N(s) \wedge \overline{\mfF}(s,X_i^N),\,
\overline\mfF_i^N(s) \vee \overline{\mfF}(s,X_i^N)\big],
\end{equation*}
with a length equal to $\big|\overline\mfF_i^N(s) - \overline{\mfF}(s,X_i^N)\big|$.
Summing over $i$ and taking expectation on the $(Q_i)_{i\le N}$,
we obtain
\begin{equation}
\mfD(t)
\le \bE_0^N\Big[\frac1N\sum_{i\in \cS^N(0)} \int_0^t|\overline\mfF_i^N(s) - \overline{\mfF}(s, X_i^N)| \Rd s\Big]\,.
\label{eq_prop_DF}
\end{equation}
Starting from \eqref{eq_prop_DF}, for any $i$ and $t$, we decompose the integrand
into seven terms  according to Lemma \ref{lem_dec_FN}.
Five among these terms are first treated 
thanks to Lemmas \ref{lem_AN_bound}, \ref{lem_VN_bound} and \ref{lem_LN_bound},
so that there exists $C>0$, 
 $\overline\mfV^N = \lambda^*\cdot \sqrt{\Upsilon^N	+ \omega^*\cdot \bar \gamma^N}$
 	 and $\overline\mfL^N = \lambda^*\cdot \Ninf{\bar \omega^N	- \bar \omega}$,
 	the later two converging in probability to zero by virtue respectively of Assumptions~\ref{hyp-Var}
 	and \ref{hyp-w},
 	 such that 
\begin{equation}
\mfD(t)
\le C\cdot \int_0^t\mfD(s)
+ t\cdot (2\overline\mfV^N + \overline\mfL^N) + \int_0^t  \langle \bar \mu^{S, N}_0, |\overline \mfE^{N, 1}(s, .)| + |\overline \mfE^{N, 0}(s, .)|\rangle \Rd s\,.
\label{eq_prop_Gronw}
\end{equation}
Thanks to Lemma \ref{lem_EN_bound},
the integrand $(\langle \bar \mu^{S, N}_0, |\overline \mfE^{N, 1}(s, .)| + |\overline \mfE^{N, 0}(s, .)|\rangle)_{s\in [0, t]}$
converges to 0 pointwise in probability for any $s>0$.
By virtue of Assumptions~\ref{hyp-Blambda}
and \ref{hyp_ws}, for the defining properties of $\lambda^*$ and $\omega^*$, we obtain that 
\begin{equation*}
|\overline \mfE^{N, 1}(s, x)| \vee |\overline \mfE^{N, 0}(s, .)| \le \omega^*\, \lambda^*,
\end{equation*}
holds for any $s\ge 0$ and any $x\in \bX$. 
Since $\bar\mu_0^{S, N}(\bX)\le 1$,
the integrand is itself upper-bounded by $\omega^*\, \lambda^*$. 
By Lebesgue's dominated convergence theorem,
we deduce the convergence to 0 in probability as $N$ tends to infinity
of the following random variable for any $t$, the r.v. being non-decreasing with $t$:
\begin{equation*}
\int_0^t  \langle \bar \mu^{S, N}_0, |\overline \mfE^{N, 1}(s, .)| + |\overline \mfE^{N, 0}(s, .)|\rangle \Rd s\,.
\end{equation*}
With \eqref{eq_prop_Gronw}, we are therefore in situation to apply
Gronwall's inequality and conclude Proposition \ref{prop_cvg_DN}
in that $\mfD(t)$ converges to 0
in probability for any $t$
as $N\to \infty$.
\end{proof}

Before proceeding, 
let us state a result 
which is exactly Theorem~II.4.1 in \cite{perkins}
and will be needed in the next proof.
	Let us generally consider a Polish space $\cY$.
We say that a subset $\mfM$ of $\cC_b(\cY)$
is separating if $(i)$ it includes the constant function $:x\mapsto 1$ and $(ii)$ 
it discriminates elements of $\cM_1(\cY)$,
in that for any $(\nu, \nu') \in \cM_1(\cY)^2$,
$\nu = \nu'$ is equivalent to the property
that $\langle \nu, \phi\rangle = \langle \nu, \phi\rangle$ for any $\phi \in \mfM$.
It is well known that $\cC_b(\cY)$ itself is separating.
A sequence of random elements of $\bD(\mathbb{R}_+, \cM(\cY))$  is said to be C-tight if it is tight and any limit of a converging subsequence is a.s. continuous.
\begin{prop}\label{prop_perkins}
	A sequence of processes
	$(\nu^N)_{n\in \mathbb{N}^*}$ is C-tight  in $\bD(\mathbb{R}_+, \cM_1(\cY))$ if and only if:
	\begin{enumerate}[(a)]
		\item\label{it:CCC} \textbf{Compact Containment Condition (CCC).} For all
		$\varepsilon>0$ and $T$, there exists a compact set $K_{\varepsilon}$ in $\cY$ such that:
		\[
		\sup_{N\in \mathbb{N} ^*} \mathbb{P}\left(\sup_{t\le T} \nu^N_t (K_{
			\varepsilon}^c) >\varepsilon\right)<\varepsilon.
		\]
		
		\item  \label{it:proj}  \textbf{Tightness  of the  projections.}
		The sequence
		$(\langle \nu^N_{\,\,\,\cdot},  \varphi \rangle)_{N\in \mathbb{N}^*}$  is C-tight
		in $\bD(\mathbb{R}_+, \mathbb{R})$
		for any function $\varphi$ in a separating class $\mfM$. 
	\end{enumerate}
\end{prop}

We can then conclude to the convergence of the measure 	$\bar \mu^{S,N}$:
\begin{prop}\label{prop_cvg_muSN}
	As $N\to\infty$, the following convergence holds  in probability 
	\[
	\bar \mu^{S,N}\to \bar\mu^S \quad\mbox{in}\quad \bD(\RR_+, \cM(\bX))\,.
	\]
\end{prop}
\begin{proof}
We apply Proposition~\ref{prop_perkins}
with $\cY = \bX$, $\nu^N= \bar \mu^{(S, N)}$ and $\mfM = \cC_{b+}(\bX)$
the set of bounded continuous functions from $\bX$ to $\RR_+$. 
$\cC_{b+}(\bX)$ is separating 
	since $\cC_{b}(\bX)$ is itself separating. 
Point $(a)$ follows readily from the two facts:  $\bar\mu^{(S, N)}_t (K_{\varepsilon}^c)\le\bar\mu^{(S, N)}_0 (K_{\varepsilon}^c)$,
and $\bar\mu^{(S, N)}_0\Rightarrow\bar\mu^{S}_0$ in probability
(exploiting for instance the Lévy-Prokhorov metric).

Concerning Point $(b)$,
let us consider any $\varphi\in \cC_{b,+}(\bX)$.
Recall that from Lemma~\ref{lem_cvg_wmu_S}  $\langle \wtd\mu^{(S, N)}_{t},  \varphi \rangle$ converges in probability to $\langle \bar\mu^{S}_{t},  \varphi \rangle$ 
for any $t>0$.
Moreover, by definition of $\wtd\mu^{(S, N)}$ and $\mfD$ in respectively \eqref{eq_def_tmuS} and \eqref{eq_def_mfD}:
\begin{equation}
\big|\langle \bar\mu^{(S, N)}_{t},  \varphi \rangle
- \langle \wtd\mu^{(S, N)}_{t},  \varphi \rangle\big|
\le \Ninf{\varphi}\cdot \mfD(t),
\end{equation}
with an upper-bound that converges to 0  in probability 
locally uniformly in $t$ thanks to Proposition \ref{prop_cvg_DN}.
Therefore, $\langle \bar\mu^{(S, N)}_{t},  \varphi \rangle$ converges in probability to $\langle \bar\mu^{S}_{t},  \varphi \rangle$ 
pointwise for any $t>0$.
In addition, since $\varphi$ is non-negative,
$t\to\langle \bar\mu^{(S, N)}_{t},  \varphi \rangle$ is non-increasing
for each $N\ge1$.
Also, $t\to\langle \bar\mu^{S}_{t},  \varphi \rangle$ is continuous.
Thanks to Lemma~\ref{lem_loc_unif_cvg}
in Appendix~\ref{sec:appendix},
the convergence in probability of $\langle \bar\mu^{(S, N)}_{t},  \varphi \rangle$ 
to $\langle \bar\mu^{S}_{t},  \varphi \rangle$
therefore holds locally uniform in $t$.

Thanks to Proposition~\ref{prop_perkins}, the sequence
$(\bar\mu^{(S, N)})_{N\in \mathbb{N}^*}$ is thus C-tight.
By the convergence of the projection, 
any limit point is necessarily $\bar\mu^{(S)}_{\,\,\,\cdot}$,
which concludes that $\bar\mu^{(S, N)}$ converges to $\bar\mu^{(S)}$ in $\bD(\RR_+, \cM(\bX))$. 
\end{proof}

The above arguments directly entail the following
pointwise in time convergence result on the force of infection. 
\begin{prop}\label{prop_cvg_mF}
The following convergence to 0 
 holds a.s. for any $t>0$:
\begin{equation}
	\lim_{N\to \infty} 
	\bE_0^N\Bigg[ \frac1N	\sum_{i\in \cS^N(0)}	\Big|\overline\mfF_i^N(t) - \overline\mfF(t, X_i^N) \Big|\Bigg]
	=0.
	\label{eq_cvg0_DF}
\end{equation}
\end{prop}

\begin{proof}
As a consequence of Lemmas \ref{lem_dec_FN}, 
\ref{lem_AN_bound}, \ref{lem_VN_bound},
\ref{lem_LN_bound} and \ref{lem_EN_bound}, we obtain
\begin{equation*}
\bE_0^N\Bigg[ \frac1N	\sum_{i\in \cS^N(0)}	\Big|\overline\mfF_i^N(t) - \overline\mfF(t, X_i^N) \Big|\Bigg]
\le C\cdot \mfD(t) 
+ 3 \overline \mfV^N
+ \lambda^*\cdot \Ninf{\bar \omega^N	- \bar \omega}
+ \overline\mfE^N(t)\,
\end{equation*}
where $C, \lambda^*<\infty$ while
$\mfD(t)$,
$\overline \mfV^N$,
$\Ninf{\bar \omega^N	- \bar \omega}$
and $\overline\mfE^N(t)$
all tend to 0 
thanks respectively to Proposition \ref{prop_cvg_DN},
 Assumption~\ref{hyp-Var},
 Assumption~\ref{hyp-w},
and  Lemma~\ref{lem_EN_bound}.
This concludes the proof of Proposition~\ref{prop_cvg_mF}.
\end{proof}

	As a consequence of Propositions~\ref{prop_cvg_DN} and \ref{prop_cvg_mF},
by exploiting a similar approach as for the convergence to $\bar\mu^S_t$
in Lemma~\ref{lem_cvg_wmu_S},
we could typically prove the following pointwise in time convergence in probability
in terms of any test function $\varphi \in C_b(\bX)$:
\begin{equation*}
\sum_{i\le N} \overline\mfF_i^N(t)\, \idc{D_i^N(t)= 0}\, \varphi(X^N_i)
\to \langle \overline \mfF(t) \bar\mu^S_t, \varphi \rangle\,.
\end{equation*}
The test function $\varphi$ evaluates here the convergence 
	of a distribution on $\bX$ that we call the activated force of infection.
The measure on the left-hand side can be interpreted as minus the derivative 
of the process $\bar \mu^{N,S}$ at time $t$.

\medskip

\subsection{Convergence of $(\bar \mu^{I, N}_{\cdot}, \bar \mu^{R, N}_{\cdot})$}

Because the proof is simpler and more related to the one of Proposition \ref{prop_cvg_muSN}
we first justify in the next proposition
the convergence of the LLN-scaled recovered process $\bar\mu^{R,N}_t$,
before we treat similarly in Proposition \ref{prop_cvg_muIN}
the process  $\bar\mu^{I,N}_t$.
\begin{prop}\label{prop_cvg_muRN}
	As $N\to\infty$, the following convergence holds  in probability 
\[
\bar \mu^{R,N}\to \bar\mu^R \quad\mbox{in}\quad \bD(\RR_+, \cM(\bX))\,.
\]
\end{prop}

\begin{proof}
We distinguish three components depending on the initial condition 
of the individuals:
\begin{equation}\label{eq_dec_muRN}
\bar\mu^{R,N}_t=   \bar\mu^{R,N}_0 + \bar\mu^{R,N,0}_t + 	\bar\mu^{R,N,1}_t\,,
\end{equation}
where the measures $\bar\mu^{R,N,0}$  and $\bar\mu^{R,N,1}$ act as follows
on test functions $\varphi \in C_b( \RR_+)$ and time $t\ge0$:
\begin{equation}	\label{eq_def_barmuRNt0}
	\langle\bar\mu^{R,N,0}_t, \varphi\rangle =  \frac{1}{N}	\sum_{j\in \cI^N(0)} \idc{\eta_j^{N, 0} \le t} \varphi(X^N_j) \,,
\end{equation}
and 
\begin{equation}	\label{eq_def_barmuRNt1}
	\langle\bar\mu^{R,N,1}_t, \varphi\rangle = \frac{1}{N}
	\sum_{i\in \sD^N(t)} \idc{\tau^N_i +\eta^N_i\le t} \varphi(X^N_i)\,. 
\end{equation}
We recall  that $\sD^N(t)$ is defined in \eqref{eq_def_DN} as the subset in $\cS^{N}_0$ 
	of individuals infected by the disease  by time $t$.
The convergence for the first term $\langle \bar\mu^{R,N}_0, \varphi \rangle$ to $\langle \bar\mu^{R}_0, \varphi \rangle$ is part of Assumption \ref{hyp-initial}. 

Concerning $\bar\mu^{R,N,0}$, 
we will justify the convergence in probability 
through the computation of the expectations and variances,
conditional on $\cF^N_0$.
\begin{equation}
\bE_0^N\Big[\langle\bar\mu^{R,N,0}_t, \varphi\rangle\Big]
= \int_0^\infty\hcm{-0.3} \int_{\bX} \varphi(x)  \bigg(1-\frac{F^c_x(a+t)}{F^c_x(a)} \bigg)
\bar\mu_0^{I, N}(\Rd x,\Rd a) 
\end{equation}
Thanks to Assumption \ref{hyp-initial}, the above conditional expectation
converges in probability to
\begin{equation}\label{eq_def_bar_muR1}
\langle\bar\mu^{R,0}_t, \varphi\rangle
=\int_0^\infty\hcm{-0.3} \int_{\bX} \varphi(x)  \bigg(1-\frac{F^c_x(a+t)}{F^c_x(a)} \bigg)
\bar\mu_0^{I}(\Rd x,\Rd a)\,. 
\end{equation}
By the independence of the $\eta_j^{N, 0}$ for $j\in \cI^N(0)$ conditionally on $\cF^N_0$, we obtain 
\begin{equation}
\begin{split}
	 \Var_0^N\Big[\langle\bar\mu^{R,N,0}_t, \varphi\rangle\Big]
&= \frac1N \int_0^\infty\hcm{-0.3} \int_{\bX} \varphi(x)^2\cdot  \frac{F^c_x(a+t)}{F^c_x(a)}\cdot \bigg(1-\frac{F^c_x(a+t)}{F^c_x(a)} \bigg)
\bar\mu_0^{I, N}(\Rd x,\Rd a)
\\&\quad \le \frac{\Ninf{\varphi}^2}N\,.
\end{split}
\end{equation}
Finally we conclude the convergence in probability of $\bar\mu^{R,N,0}_t$ to $\bar\mu^{R,0}_t$
as defined in \eqref{eq_def_bar_muR1}.
\bigskip

We then look at the mean-field approximation of $\bar\mu^{R,N,1}$ as defined in \eqref{eq_def_barmuRNt1}:
\begin{equation}	\label{eq_def_td_muRNt1}
	\langle\wtd\mu^{R,N,1}_t, \varphi\rangle = \frac{1}{N}
	\sum_{i\in \wtd \sD^N(t)} \idc{\wtd\tau^N_i +\eta^N_i\le t} \varphi(X^N_i)\,. 
\end{equation}
We recall that $\wtd \sD^N(t)$
 is defined in \eqref{eq_def_wDN} as the subset
of individuals infected according to $(\wtd \tau_i^N)$
during the time-interval $(0,t]$.
\begin{equation}\label{eq_eq_exp_td_muRN}
	\bE_0^N\Big[\langle\wtd\mu^{R,N,1}_t, \varphi\rangle
\Big]
	=  \int_{\bX} \varphi(x) \int_0^t  
\overline \mfF(s, x) \cdot \exp\bigg[-\int_0^s \overline \mfF(r, x) \Rd r\bigg]
	\cdot F_x(t-s)
	\Rd s\,
	\bar\mu_0^{S, N}(\Rd x)\;.
\end{equation}
Thanks to Assumption \ref{hyp-initial} and \eqref{eq_def_bmuS2}, the above conditional expectation
converges in probability to:
\begin{equation}\label{eq_def_bar_muR0}
\begin{split}
		\langle\bar\mu^{R,1}_t, \varphi\rangle
	&= \int_0^t\hcm{-0.2} \int_{\bX} \varphi(x)\cdot F_x(t-s)\cdot 
	\overline \mfF(s, x) \cdot \exp\bigg[-\int_0^s \overline \mfF(r, x) \Rd r\bigg]
	\bar\mu_0^{S}(\Rd x)\,\Rd s
	\\&= \int_0^t\hcm{-0.2} \int_{\bX} \varphi(x)\cdot F_x(t-s)\cdot 
	\overline \mfF(s, x)
	\bar\mu_s^{S}(\Rd x)\,\Rd s\,.
\end{split}
\end{equation}
By the independence of the $\eta^N_i$ and of the $\wtd\tau_i^N$ for $i\in \cS^N(0)$ conditionally on $\cF^N_0$, we obtain 
\begin{equation}\label{eq_var_td_muRN}
	\begin{split}
		\Var_0^N\Big[\langle\wtd\mu^{R,N,1}_t, \varphi\rangle\Big]
		&\le \frac{\Ninf{\varphi}^2}N\,.
	\end{split}
\end{equation}
Recalling \eqref{eq_def_mfD},
since the two events $\{\wtd \tau_i^N + \eta^N_i \le t\}$
and $\{\tau_i^N + \eta^N_i \le t\}$ agree 
on the event $\{D^N_i(r) = \wtd D^N_i(r), \forall r\in [0, t]\}$:
\begin{equation}\label{eq_diff_tdbar_muRN0}
|\langle\wtd\mu^{R,N,1}_t - \bar \mu^{R,N,1}_t, \varphi\rangle|
\le \Ninf{\varphi}\cdot \mfD(t).
\end{equation}
The right-hand side converges in probability to 0 as $N\to \infty$  thanks to Proposition~\ref{prop_cvg_DN}.
With \eqref{eq_eq_exp_td_muRN},
\eqref{eq_def_bar_muR0}, \eqref{eq_var_td_muRN}
and \eqref{eq_diff_tdbar_muRN0} we deduce the convergence in probability of $\langle\wtd\mu^{R,N,1}_t, \varphi\rangle$ to $\langle\wtd\mu^{R,1}_t, \varphi\rangle$
as defined in \eqref{eq_def_td_muRNt1}.
By recalling \eqref{eq_dec_muRN},
the convergence of $\bar\mu^{R,N,0}_t$ and \eqref{eq_def_bmuR},
we conclude the convergence in probability of  $\langle\bar\mu^{R,N}_t, \varphi\rangle$ to $\langle\bar\mu^{R}_t, \varphi\rangle$.

For non-negative $\varphi$, the function	$:t\mapsto \langle\bar\mu^{R,N}_t, \varphi\rangle$
is non-decreasing. We can thus easily adapt the argument  given in Proposition  \ref{prop_cvg_muSN},
which notably involves the tightness criteria given in Proposition~\ref{prop_perkins}
with $\nu^N = \bar \mu^{R, N}$,
still $\cY = \bX$ and $\mfM = \cC_{b+}(\bX)$
so as to apply  the Second Dini theorem.
For any $t$ and any compact set $K$, $\bar\mu^{R,N}_t(K^c) \le \bar \mu^N_X(K^c)$
as defined in \eqref{eq_def_muXN}. Since the latter converges in probability to $\bar \mu_X$,
recalling \eqref{eq_def_muXN},
we can find for any $\eps>0$ some compact set $K_\eps$ such that 
\begin{equation*}
\sup_{N\ge 1} \PR\Big(\bar \mu^N_X(K_\eps^c)> \eps\Big)< \eps.
\end{equation*}
Point $(a)$  in Proposition~\ref{prop_perkins} is therefore verified as well.
So we deduce that the convergence in probability
extends to the function $\bar\mu^{R,N}$ in $\bD_1$,
which concludes the proof of Proposition~\ref{prop_cvg_muRN}.
\end{proof}

\begin{prop}\label{prop_cvg_muIN}
	As $N\to\infty$, the following convergence holds  in probability 
	\[
	\bar \mu^{I,N}\to \bar\mu^I \quad\mbox{in}\quad \bD(\RR_+, \cM(\bX\times \RR_+))\,.
	\]
\end{prop}
\begin{proof}
Although the proof is more technical
	than  the one of Proposition \ref{prop_cvg_muRN},
	we exploit similar arguments. 
In order to exploit the Second Dini theorem, 
we wish to consider non-increasing projections.
This is why we will study  the following extended measure $\bar \mu^{SI, N}_t$
 for test functions $\psi$ 
 in the set $\mathfrak M(\bX\times [-1, \infty))$ of  functions from $\bX\times [-1, \infty)$ to $\RR_+$
 that are continuous, bounded, non-negative and non-increasing in the second variable:
\begin{align*}
\langle	\bar \mu^{SI, N}_t, \psi \rangle
&:= \langle	\bar \mu^{S, N}_t, \psi(\,.\,, -1) \rangle
+ \langle	\bar \mu^{I, N}_t, \psi_{|\bX\times [0, \infty)} \rangle
\\&= \frac1N \sum_{i\in \cS^{N}(0)} \bigg(\idc{D_i^N(t) = 0} \psi(X_i^N, -1)
+ \idc{D_i^N(t) = 1} \,\idc{\eta^N_i > t- \tau_i^N }\psi(X_i^N, t- \tau_i^N) \bigg)
\\
&\quad+ \frac1N  \sum_{j\in \cI^{N}(0)}\,\idc{\eta_i^{N, 0} > t} \psi(X_j^N, A_j^N(0)+t)\,.
\end{align*}
In words, $\bar \mu^{SI, N}$ is derived from the addition of both $\bar \mu^{S, N}$ and $\bar \mu^{I, N}$ where the first measure on $\bX$ is projected with a fixed component $-1$ according to the age variable. 
Intuitively, what we are doing is considering susceptible individuals as infected 
with infection age $-1$.
The fact that $\psi$ is non-negative implies that any recovery event leads to a reduction
of $\langle	\bar \mu^{SI, N}, \psi \rangle$ at this particular time.
The fact that $\psi$ is non-increasing in the second variable implies that the aging of the actively infected  leads as well to a reduction of $\langle	\bar \mu^{SI, N}, \psi \rangle$
over time.
With this trick of combining  $\bar \mu^{S, N}$ to  $\bar \mu^{I, N}$ into $\bar \mu^{SI, N}$, 
any infection event leads as well to a  reduction of $\langle	\bar \mu^{SI, N}, \psi \rangle$ at the infection time.

We will make use of Proposition~\ref{prop_perkins}
in combination with 
the following lemma
by considering for $\nu^N = \bar \mu^{SI, N}$
the set $\cY = \bX \times [-1, \infty)$
and the proposed set $\mfM(\bX\times [-1, \infty))$ as  the separating class.
\begin{lem}\label{lem_sep_mfM}
The set $\mfM(\bX\times [-1, \infty))$ 
of functions that are continuous, bounded, non-negative and non-increasing in the second variable
 is a separating class.
\end{lem}
\begin{proof}
The fact that the constant function equal to 1 is part of $\mfM(\bX\times [-1, \infty))$ comes readily from the definition.
If $\nu, \nu'$ are such that $\langle \nu, \phi\rangle = \langle \nu', \phi\rangle$ for any $\phi \in \mfM$,
then the classical approximation scheme of indicator functions by bounded continuous functions
leads to the identity $\nu(A\times [-1, a]) = \nu'(A\times [-1, a])$,
for any $a\in [-1, \infty)$ and measurable subset $A$ of $\bX$.
The sets of this form $A\times [-1, a]$ form a $\pi$-system 
of subsets of the product space $\bX\times [-1, \infty)$
that contains $\bX\times [-1, \infty)$ itself and
generates the Borel $\sigma$-field of $\bX\times [-1, \infty)$.
The identity $\nu = \nu'$ is thus deduced thanks e.g. to \cite[Lemma 1.17]{Ka02}.
This concludes the proof of Lemma~\ref{lem_sep_mfM}.
\end{proof}
Then, it mainly remains to adapt the computations of conditional expectations and variances
from the proof of Proposition \ref{prop_cvg_muRN}.
$\bar \mu^{I, N}$ is similarly decomposed into the sum of $\bar \mu^{I, N, 0}$ and $\bar \mu^{I, N, 1}$, 
that are represented as follows for any $\psi \in C_b(\bX\times \RR_+)$ and $t\ge 0$:
\begin{equation}	\label{eq_def_bar_muINt0}
	\langle \bar\mu^{I,N,0}_t, \psi\rangle  = \frac{1}{N} \sum_{j\in \cI^N(0)}   \idc{\eta_j^{N, 0} >t}  \psi(X^N_j, A_j^N(0)+t)\,,
\end{equation}
and
\begin{equation}	\label{eq_def_barmuINt1}
	\langle\bar \mu^{I,N,1}_t, \psi\rangle =  \frac{1}{N}  \sum_{i\in \sD^N(t)} \idc{\tau^N_i +\eta^N_i>t} \psi(X^N_i, t-\tau^N_i)\,.
\end{equation}
Concerning $\bar\mu^{I,N,0}$, we have 
\begin{equation*}
	\bE_0^N\Big[\langle\bar\mu^{I,N,0}_t, \psi\rangle\Big]
	= \int_0^\infty\hcm{-0.3} \int_{\bX} \psi(x, a + t)  \cdot \frac{F^c_x(a+t)}{F^c_x(a)} 
	\bar\mu_0^{I, N}(\Rd x,\Rd a) 
\end{equation*}
Thanks to Assumption \ref{hyp-initial}, the above conditional expectation
converges in probability to:
\begin{equation}\label{eq_def_bar_muI1}
	\langle\bar\mu^{I,0}_t, \psi\rangle
	=\int_0^\infty\hcm{-0.3} \int_{\bX} \psi(x, a+t)  \cdot\frac{F^c_x(a+t)}{F^c_x(a)}\,
	\bar\mu_0^{I}(\Rd x,\Rd a)\,. 
\end{equation}
By the independence of the $\eta_j^{N, 0}$ for $j\in \cI^N(0)$ conditionally on $\cF^N_0$, we obtain 
\begin{equation*}
	\begin{split}
		\Var_0^N\Big[\langle\bar\mu^{I,N,0}_t, \psi\rangle\Big]
		&= \frac1N \int_0^\infty\hcm{-0.3} \int_{\bX} \psi(x, a+t)^2\cdot  \frac{F^c_x(a+t)}{F^c_x(a)}\cdot \bigg(1-\frac{F^c_x(a+t)}{F^c_x(a)} \bigg)
		\bar\mu_0^{I, N}(\Rd x,\Rd a)
		\\&\quad \le \frac{\Ninf{\psi}^2}N\,.
	\end{split}
\end{equation*}
So we conclude to the convergence in probability of $\langle\bar\mu^{I,N,0}_t, \psi\rangle$ to $\langle\bar\mu^{I,N,0}_t, \psi\rangle$
as defined in \eqref{eq_def_bar_muI1}, 
valid for any $t>0$ and any $\psi\in \cC_b(\bX\times \RR_+)$.
\medskip

We then consider the mean-field approximation of $\bar\mu^{I,N,1}$ as defined in \eqref{eq_def_barmuINt1}, exploiting the notations $\wtd \tau_i^N$ and $\wtd\sD^N(t)$ from \eqref{eq_def_wDN}:
\begin{equation}	\label{eq_def_td_muINt1}
	\langle\wtd\mu^{I,N,1}_t, \psi\rangle = \frac{1}{N}
	\sum_{i\in \wtd\sD^N(t)} \idc{\wtd\tau^N_i +\eta^N_i> t} \psi(X^N_i, t - \wtd\tau^N_i)\,. 
\end{equation}
\begin{equation}\label{eq_eq_exp_td_muIN}
	\bE_0^N\Big[\langle\wtd\mu^{I,N,1}_t, \psi\rangle\Big]
	=  \int_{\bX} \int_0^t  
 	\overline \mfF(s, x) \cdot \exp\bigg[-\int_0^s \overline \mfF(r, x) \Rd r\bigg]
	\cdot F_x(t-s)\cdot \psi(x, t-s)
	\Rd s\,
	\bar\mu_0^{S, N}(\Rd x)\;.
\end{equation}
Thanks to Assumption \ref{hyp-initial} and \eqref{eq_def_bmuS2}, the above conditional expectation
converges in probability to
\begin{equation}\label{eq_def_bar_muI0}
	\begin{split}
		\langle\bar\mu^{I,1}_t, \psi\rangle
		&=\int_{\bX} \hcm{-0.1}  \int_0^t\psi(x, t-s)\cdot F_x(t-s)\cdot 
		\overline \mfF(s, x) \cdot \exp\bigg[-\int_0^s \overline \mfF(r, x) \Rd r\bigg]
	\Rd s\,	\bar\mu_0^{S}(\Rd x)
		\\&= \int_0^t\hcm{-0.2} \int_{\bX} \psi(x, t-s)\cdot F_x(t-s)\cdot 
		\overline \mfF(s, x)
		\bar\mu_s^{S}(\Rd x)\,\Rd s\,.
	\end{split}
\end{equation}
By the independence of the $\eta^N_i$ and of the $\wtd\tau_i^N$ for $i\in \cS^N(0)$ conditionally on $\cF^N_0$, we obtain
\begin{equation}\label{eq_var_td_muIN}
	\begin{split}
		\Var_0^N\Big[\langle\wtd\mu^{I,N,1}_t, \psi\rangle\Big]
		&\le \frac{\Ninf{\psi}^2}N\,.
	\end{split}
\end{equation}
Recalling \eqref{eq_def_mfD},
since the two events $\{\wtd \tau_i^N + \eta^N_i > t\}$
and $\{\tau_i^N + \eta^N_i > t\}$ agree 
on the event $\{D^N_i(r) = \wtd D^N_i(r), \forall t\in [0, t]\}$, we get
\begin{equation}\label{eq_diff_tdbar_muIN0}
	|\langle\wtd\mu^{I,N,1}_t - \bar \mu^{I,N,1}_t, \psi\rangle|
	\le \Ninf{\psi}\cdot \mfD(t).
\end{equation}
The right-hand side converges in probability to 0 as $N$ tends to infinity thanks to Proposition~\ref{prop_cvg_DN}.
With \eqref{eq_eq_exp_td_muIN},
\eqref{eq_def_bar_muI0}, \eqref{eq_var_td_muIN}
and \eqref{eq_diff_tdbar_muIN0} we deduce the convergence in probability of $\langle\wtd\mu^{I,N,1}_t, \psi\rangle$ to $\langle\wtd\mu^{I,1}_t, \psi\rangle$
as defined in \eqref{eq_def_bar_muI0}.

This concludes the proof that $\langle\bar\mu^{I,N}_t, \psi\rangle$
converges in probability to $\langle\bar\mu^{I}_t, \psi\rangle$,
for any $t>0$ and any $\psi \in \cC_b(\bX\times \RR_+)$.
Recalling Proposition~\ref{prop_cvg_muSN},
we deduce specifically that $\langle\bar\mu^{SI,N}_t, \psi\rangle$
converges in probability to $\langle\bar\mu^{SI}_t, \psi\rangle$,
for any $t>0$ and any $\psi \in \mfM$,
where 
\begin{equation}
\langle\bar\mu^{SI}_t, \psi\rangle
:= \langle	\bar \mu^{S}_t, \psi(\,.\,, -1) \rangle
+ \langle	\bar \mu^{I}_t, \psi_{|\bX\times [0, \infty)} \rangle\,.
\end{equation}
Note about the above definition of $\bar\mu^{SI}_t$ through test functions $\psi \in \mfM$
that it already uniquely specifies $\wtd\mu^{SI}_t$ due to $\mfM$ being a separating class.
The extension of this definition to any $\psi \in \cC_b(\bX \times [-1, \infty))$
is yet very natural.

With the crucial arguments given at the beginning of this proof of Proposition \ref{prop_cvg_muIN},
recall that $\langle\bar\mu^{SI,N}_t, \psi\rangle$ is for any $\psi\in \mfM$ non-increasing as a function of $t$.
On the other hand, $\langle\bar\mu^{SI}_t, \psi\rangle$ is deterministic, continuous and non-increasing as a function of $t$.
We can thus adapt the argument  given in Proposition~\ref{prop_cvg_muSN}
to show that
the convergence in probability of $\langle\bar\mu^{SI,N}_t, \psi\rangle$
to $\langle\bar\mu^{SI}_t, \psi\rangle$
is locally uniform in $t$.
This concludes Point $(b)$ in Proposition~\ref{prop_perkins}.
Concerning Point $(a)$,
we remark for any compact set $K$ in $\bX$ and any $A>0$ that 
\begin{equation}\label{eq_prop_mu_tight}
	\bar\mu^{SI,N}_t\big[(K\times [-1, A+t])^c\big]
\le \bar\mu^{S,N}_0\big[K^c\big]
+ \bar\mu^{I,N}_0\big[(K\times [0, A])^c\big].
\end{equation}
Since $\bar\mu^{S,N}_0$ and  $\bar\mu^{I,N}_0$ converge in probability to respectively $\bar\mu^{S}_0$ and  $\bar\mu^{I}_0$,
there exists for any $\eps>0$ such a compact set $K$ in $\bX$ and $A>0$ that satisfy 
\begin{equation*}
\sup_N \PR(\bar\mu^{S,N}_0\big[K^c\big]
+ \bar\mu^{I,N}_0\big[(K\times [0, A])^c\big] > \eps)
< \eps.
\end{equation*}
Recalling \eqref{eq_prop_mu_tight},
this entails Point $(a)$,
 for any $T>0$ with $K^{(\cY)} = K\times [-1, A + T]$.
We then exploit Proposition~\ref{prop_perkins}
and conclude the proof of Proposition~\ref{prop_cvg_muIN}
that $\bar \mu^{(I, N)}$,
as the restriction of $\bar \mu^{(SI, N)}$ to $\bX \times \RR_+$,
 converges in probability to $\bar \mu^{(I, N)}$
 in $\bD(\RR_+, \cM(\bX\times \RR_+))$.
 \end{proof}

\section{Concluding remarks}\label{sec_disc_MT}

Typical examples of pairwise kernel interactions found in applications
are contact matrices over a discrete space as in \cite{BBT2020}, e.g., as well as spatial interactions as in  \cite{KP2024spatial}
with characteristics that are distributed on a compact subset of $\RR^d$.
Our conditions encompass both settings,
and any combination of those.

They also allow for a scaling factor $\eps^N>0$ 
depending on $N$
such that $\bar \kappa^N = \eps^N \cdot \bar \kappa$
where $\bar \kappa$ can be a constant but also a non-degenerate bounded measurable function.
A typical choice is for instance $\eps^N = N^{-\alpha}$ with $\alpha \in (0, 1)$.
Provided we assume as in \cite{KHT2022}
that the contact rate  is fixed  at some value $\bar \gamma^N>0$, 
$N\eps^N \to \infty$ is required (and sufficient)
for Assumption~\ref{hyp-Var} to hold. 
We do not need to require the convergence to infinity of $N\eps^N/\log(N)$ as in \cite{KHT2022}.
We remark that $N\eps^N\to \infty$ ensures that the node degrees diverge, with no condition on their relation to $N$.
We refer to \cite{DFG+24}
for more details on the relation to the denseness level of the graph
(in terms notably of the number of edges or the node degrees).
The case where $\eps^N = N^{-1}$
leads to limiting equations of a different nature,
as hinted by the exploratory simulations in 
\cite{DFG+24}
and clarified in the results of \cite{BHI2024} for 
specific cases of stochastic block models. 

Besides, our conditions guarantee the robustness of the convergence against perturbations that concern a negligible subset of edges,
since the two statistics 
introduced in Assumption~\ref{hyp-Var}
are averages.

Moreover, 
it can be helpful to introduce
different kinds of interactions
 with their corresponding scaling factors.
For example, 
we may consider the $(X_i^N)$ as spatial coordinates, 
say uniformly distributed in $[0, 1]$ 
and with the norm-distance $d(x, y) = |x-y|$  on the circle
so that all the locations are equivalent
(and where $0$ is merged with $1$).
Let us then distinguish in the expressions of $\kappa^N$ and $\gamma^N$
a local interaction pattern (specified by the index $\cL$) with radius $\delta \in (0, 1/2)$
and a global interaction pattern (with index $\cG$) over the whole domain: 
\begin{equation*}
	\kappa^N(x, y) = 
	\begin{cases}
		\kappa^{N, \cL}
		\\
		\kappa^{N, \cG}
	\end{cases},
	\quad \gamma^N(x, y) = 
	\begin{cases}
		\gamma^{N, \cL}
		\quad \text{ if } |x - y| \le \delta,
		\\
		\gamma^{N, \cG}
		\quad \text{ otherwise.}
	\end{cases}
\end{equation*}
Typically, we could expect $\kappa^{N, \cL} \gg \kappa^{N, \cG}$ 
while $\gamma^{N, \cG} \gg \gamma^{N, \cL}$,
that is,  many small local contacts 
as compared to rare yet strongly connected global contacts.
Such a framework on a network 
is captured by our model.

Both types of contacts remain in the limit
provided both $\kappa^{N, \cL} \cdot \gamma^{N, \cL}$
and $\kappa^{N, \cG} \cdot \gamma^{N, \cG}$
scale as $N^{-1}$.
Actually we see that the contribution of the local contacts 
then outcompetes the one of global contacts
under the following condition:
$(\kappa^{N, \cL} \cdot \gamma^{N, \cL})/(\kappa^{N, \cG} \cdot \gamma^{N, \cG})
\ge (2\delta)/(1-2\delta)$.

\medskip

A possible extension of our result would be to consider individual types 
that  evolve in time, 
typically seasonal or when the individual locations move in space.
Another natural objective would be to establish the FLLN in instances where the limiting kernel $\bar \omega$ incorporates an additional dependency on the overall distribution $\bar \mu_X$ as in the general setting of \cite{KP2024spatial}.
The inclusion of superspreading events,
defined as the occurrence of highly heterogeneous transmission,
would require to relax
the boundedness conditions
(see \cite{DDZ22a} in this direction regarding $\bar\omega$
or \cite{FP26} regarding $\lambda_i^N$), 
which also leads to additional difficulties.

\appendix
\section{Technical supporting results} \label{sec:appendix}

In Appendix~\ref{sec:appendix},
we prove three technical results
that are exploited in the current paper.
Lemma~\ref{lem_loc_unif_cvg} is used to deduce local uniform convergence in probability
from pointwise estimates.
Lemma~\ref{lem_pair_ctn} 
is used to deduce the a.e. continuity 
of sections of a.e. continuous kernels
and a related result is stated in the next Lemma~\ref{lem_ae_cont}. 
The more technical proof of 
Proposition~\ref{prop_cvg_mu0}
is given afterwards.

\subsection{From pointwise to locally uniform convergence in probability}\hfill\\
In Lemma~\ref{lem_loc_unif_cvg},
we state that the second Dini theorem extends 
to convergences in probability 
of random functions.
\begin{lem}\label{lem_loc_unif_cvg}
Let $\psi$
be a possibly random non-decreasing and continuous function
from $\RR_+$ to $\RR$.
Let also $\psi^n$
be a possibly random sequence of non-decreasing functions
from $\RR_+$ to $\RR$
that converges pointwise in probability to $\psi$.
Then, $\psi^n$ converges in probability to $\psi$ locally uniformly.
\end{lem}

\begin{proof}
We exploit the relation between the convergence in probability 
and the a.s. convergence along  sequence extractions
as stated in  \cite[Lemma 4.2]{Ka02}.

Let $(N[k])_{k\ge 1}\in \mathbb{N}^\mathbb{N}$ be an increasing sequence
and for any $T$, let $(t_i)_{i\ge 1}$ be a countable dense subset of $[0, T]$.
By a triangular argument, 
we can then define an extraction $(\wtd N_n)_{n\ge 1} = (N[k_n])_{n\ge 1}\in \mathbb{N}^\mathbb{N}$,
with the sequence $(k_n)_{n\ge 1}$ being increasing,
such that for any $i\ge 1$, 
$\psi^{\wtd N_n}(t_i)$
converges a.s. to $\psi(t_i)$
as $n$ tends to infinity.
Thanks to the second Dini theorem (see Exercise 127 on page 81, and its solution on page 270 in 
Polya and Szeg{\"o} \cite{PS1978}),
this entails the a.s. convergence of $\psi^{\wtd N_n}(t)$ to $\psi(t)$
uniformly in $t\in [0, T]$.
Since the sequence $(\wtd N_n)$ is an extraction of any initial subsequence and $T$ can be freely chosen, 
this concludes thanks to \cite[Lemma 4.2]{Ka02}
that $\psi^{\wtd N_n}(t)$
converges in probability to $\psi(t)$
locally uniformly in $t$.
\end{proof}

\subsection{Almost everywhere continuity}

\begin{lem}\label{lem_pair_ctn}
Let $\cY$ be a Polish space.
Let $\mu \in \cM(\bX)$, $\nu\in \cM(\cY)$
and  $k:\bX\times \cY \mapsto \RR$ 
be a bounded measurable function
that is $(\mu \otimes \nu)$-a.e.  continuous.
Then the function $:x\mapsto \bar k(x, .)$ is
$\mu$ a.e. continuous 
from $\bX$ with values in  $L^1(\nu)$.
\end{lem}
Lemma~\ref{lem_pair_ctn}
is firstly exploited in the proof of Lemma~\ref{lem_reg_bF}
with $\cY = \bX$, $\mu = \nu = \bar \mu_X$ and $k = \bar \omega$. It is also involved in the proof of the following Lemma~\ref{lem_ae_cont}.

\begin{proof}
For any $\delta, \eta>0$, $x\in \bX$ and $x'\in B(x, \eta)$:
\begin{equation}\label{eq_prop_L1cont}
	\int_\cY |k(x, y) - k(x', y)|\,\nu(\Rd y)
	\le \delta  +  \Ninf{k}\, \nu\big(\cG_x[\delta, \eta]),
\end{equation}
where $\cG_x[\delta, \eta] = \{y\in \cY,\, \exists x''\in B(x, \eta),\, |k(x, y) - k(x'', y)|> \delta\}$.	
	
By assumption, there exists a measurable subset $N$ of $\bX\times \cY$
such that $\mu\otimes\nu(N)=0$
and such that $k$ is continuous in any $(x, y)\in \bX^2\setminus N$.
By  the Fubini-Tonelli theorem, for any $x\in \bX$,
the section $N_x$ is measurable, where $N_x := \{y\in \bX, (x, y)\in N\}$.
In addition, 
$\int_\bX \nu(N_x)\, \mu(\Rd x) 
=\mu\otimes\nu(N) = 0$,
so that $\nu(N_x) = 0$ for $\mu$ almost every $x$.
For any $y\notin N_x$, $k$ is continuous in $(x, y)$, thus $k(x, .)$ is continuous in $y$,
which entails $\cap_{n\ge 1} \cG_x[\delta, 2^{-n}]\in N_x$.
For any $x$ such that  $\nu(N_x) = 0$,
$\nu\big(\cG_x[\delta, \eta])$ tends to zero as $\eta$ tends to zero. 
Recalling \eqref{eq_prop_L1cont},
it concludes the proof of Lemma~\ref{lem_pair_ctn}.
\end{proof}
	
\begin{lem}\label{lem_ae_cont}
Let $t>0$, $\cY$ be a Polish space
and $\mu \in \cM(\cY)$.
Let the  measurable function $k:\cY\times [0, t]\to \RR$ be bounded and $\mu\otimes Leb$-a.e. continuous
and the measurable function $F:\cY\times [0, t]\to \RR$ be bounded and such that 
$:y\mapsto F(y, .)$ is $\mu$-a.e. continuous with
respect to the uniform norm in $[0, t]$.
Then, $:y\mapsto \int_0^t k(y, a)\, F(y, a)\, \Rd a$
is $\mu$-a.e. continuous.
\end{lem}
$Leb$ denotes the Lebesgue measure on $[0, t]$.
This lemma is to be applied in the proof of Lemma~\ref{lem_EN_bound}
with $\cY = \bX\times \bX$,
$\mu(\Rd y) = \bar \mu_X\otimes \bar \mu_X(\Rd x, \Rd x')$,
$k(y, a) = \bar \omega(x, x')\,\bar \lambda(x', t-a)$
and $F(y, a) = \overline{\cF}(x', a)\cdot \exp[-\int_0^a\overline{\cF}(x', r)\,\Rd r]$.

It is standard  that $k$ is bounded and $\mu\otimes Leb$-a.e. continuous under Assumption~\ref{hyp-Blambda}, \ref{hyp_ws} and \ref{hyp-w}.
Thanks to \eqref{eq_diff2_mfF}, $F$ is bounded and for any $y_1 = (x_1, x'_1)\in \cY$ and $y_2 = (x_2, x'_2)\in \cY$:
\begin{equation*}
\|F(y_1, .) - F(y_2, .)\|_{L^\infty([0, t])}
	\le (1 + \lambda^*\,\omega^*\,t)\, \|\overline{\cF}(x'_1, .) - \overline{\cF}(x'_2, .)\|_{L^\infty([0, t])}\,,
\end{equation*} 
which entails the required regularity of $F$
thanks to Lemma~\ref{lem_reg_bF}.

\begin{proof}
For any $y_1, y_2\in \cY$:
\begin{equation*}
\begin{split}
&	|\int_0^t k(y_1, a)\, F(y_1, a)\, \Rd a
- \int_0^t k(y_2, a)\, F(y_2, a)\, \Rd a|
 \\&\quad \le \Ninf{F}\,\|k(y_1, .)-k(y_2, .)\|_{L^1(Leb)}
+ t\,\Ninf{k}\, \|F(y_1, .)-F(y_2, .)\|_{L^\infty([0, t])}\,.
\end{split}
\end{equation*}
It concludes the proof of Lemma~\ref{lem_ae_cont} by virtue of the assumed regularity of $k$ and $F$,
thanks to Lemma~\ref{lem_pair_ctn} for the first term in $k$.
\end{proof}

\subsection{Proof of Proposition  \ref{prop_cvg_mu0}}

We first consider $(\mu^N)$ and $(\nu^N)$
as deterministic sequences,
and  then extend the result to random sequences in the last fifth step of the proof.

\begin{proof}
	By linearity of the above quantity in the function $k$ and in the pair $(\nu^N, \nu)$, 
	we may assume without loss of generality that $k$ is non-negative
	and bounded by 1,
while $\nu^N(\cY)\le 1$ and $\nu(\cY)\le 1$.
	Let us define the integrand in $x$ as $\eps^N(x)$:
	\begin{equation}
		\eps^N(x)
		= \big|\int_\cY k(x, y) [\nu^N - \nu](dy)\big|\,. 
		\label{eq_def_epsN}
	\end{equation}

In the degenerate case where $\nu \equiv 0$,
the convergence of $\langle\mu^N, \eps^N\rangle$ to 0
can be directly  deduced with the upper-bound of $\eps^N$ by $\Ninf{k}\cdot \nu^N(\cY)$ which converges to 0.
In the degenerate case where $\mu \equiv 0$,
it suffices to take $3\Ninf{k}\cdot \nu(\cY)$ as the uniform upper-bound of $\eps^N$
for any $N$ sufficiently large,
as $\mu^N(\bX)$ then tends to zero.
In the following, we can thus assume that both $\nu(\cY)>0$ and $\mu(\bX)>0$.

	The  irregularities of $k$ 
	will be located through the following subset $\cG[\delta, \eta]$ of $\bX\times \cY$,
	defined for any $\delta, \eta>0$:
	\begin{equation}
		\cG[\delta, \eta]
		:= \{(x,y)\in \bX\times \cY;\; Diam_k(B[(x, y); \eta])> \delta\},
		\label{eq_def_cGde}
	\end{equation}
	where the diameter function $Diam_k$ corresponding to the kernel $k$ is defined as follows for any measurable subset $A$ of $\bX\times \cY$:
	\begin{equation}
		Diam_k(A)
		:= \sup\{|k(z) - k(z')|; z, z'\in A\},	
	\end{equation}
	while $B[(x, y); \eta]$ denotes the open ball centered in $(x, y)$ of radius $\eta$.
	The size $\eta$ of the vicinities in \eqref{eq_def_cGde}
	shall be considered sufficiently small to ensure that discrepancies of order $\delta$
	are exceptional. 
	The measurability of the set $\cG[\delta, \eta]$ can be more directly verified
through the following definition, 
which happens to be equivalent to the one given in \eqref{eq_def_cGde}:
\begin{equation*}
	\cG[\delta, \eta]
	= \cup_{\{q\in \mathbb Q_+, m\ge 1\}} k^{-1}([0, q])^\eta
	\cap k^{-1}([q + \delta + 2^{-m}, \Ninf{k}])^\eta,
\end{equation*}
where $A^\eta$ is defined as follows for any Borel subset $A$ of $\bX^2$ and any $\eta>0$:
\begin{equation}\label{eq_def_Aeta}
	A^\eta := \{(x, x')\in \bX^2;\, 
	B[(x, x'); \eta] \cap A \neq \emptyset\},
\end{equation} 	
so that $A^\eta$ 
denotes the $\eta$-vicinity of $A$.

		\subsubsection*{Step 1: Convergence of $\langle \mu, \eps^N  \rangle$ to zero}
	\hfill	\\
		Since $k$ is $(\mu \otimes \nu)$-a.e. continuous,
		in particular, $k(x, .)$ is $\nu$-a.e. continuous  
		for $x$ on a measurable set $\cA\subset \bX$ such that  $\mu(\cA)= 1$.
For any $x\in \cA$,
		thanks to the Portmanteau theorem,
		see e.g. \cite{JS03}, Subsection IV.3a on the "Weak Convergence of Probability Measures",
		$\eps^N(x)$ converges to 0.
Remark as compared to the classical version of Portmanteau theorem
that we allow $\nu^N$ and $\nu$ to be general non-negative finite measure
rather than probability measures,
given that the proof is not difficult to adapt for this setting.
	As a consequence of Lebesgue's dominated convergence theorem,
	recalling that $\eps^N$ is bounded (by 1 under our assumption),
		we deduce 
		\begin{equation}
			\lim_{N\to \infty} \langle \mu, \eps^N  \rangle = 0.
			\label{eq_cvg0_muEpsN}
		\end{equation}
		
		\subsubsection*{Step 2: Convergence of $[\mu \otimes \nu](\cG[\delta, \eta])$ 
		and $[\mu \otimes \nu^N](\cG[\delta, \eta])$ to zero}
		\hfill\\
		Remark that the points of discontinuity of the kernel $k$
		are identified as follows in terms of the sets $\cG[\delta, \eta]$:
		\begin{equation*}
			\cup_{n\ge 1}\cap_{m\ge 1} \cG[2^{-n}, 2^{-m}].
		\end{equation*}
		Note also that the sets $\cG[\delta, \eta]$ are increasing 
		as $\delta$ decreases
		and non-increasing as $\eta$ decreases.
		Therefore, 
		due to the fact that $k$ is $(\mu \otimes \nu)$-a.e. continuous,
		the following convergence to zero holds for any $\delta$:
		\begin{equation}
			\lim_{\eta\to 0}	[\mu \otimes \nu](\cG[\delta, \eta])
			= 0.
			\label{eq_prop_cGde}
		\end{equation}
		Secondly, we remark that the weak convergence of $\nu^N$ to $\nu$
		implies the weak convergence of $\mu \otimes\nu^N$ to $\mu \otimes\nu$.
		For any $N$ sufficiently large, thanks to the Portmanteau theorem:
		\begin{equation}
			[\mu \otimes\nu^N](\cG[\delta, \eta])
			\le [\mu \otimes\nu](\cG[\delta, \eta]^{\eta}) + \eta,
			\label{eq_prop_cvg_Gde}
		\end{equation}
		where we recall the notation $A^\eta$ from \eqref{eq_def_Aeta}.
		Since $B[(x, y); \eta] \subset B[(x', y'); 2\eta]$
		holds true for any $(x', y') \in B[(x, y); \eta]$,
		it is a straightforward consequence of definition 
		\eqref{eq_def_cGde}
		that $\cG[\delta, \eta]^{\eta} \subset \cG[\delta, 2\eta]$.
		Recalling \eqref{eq_prop_cGde}
		and coming back to \eqref{eq_prop_cvg_Gde},
		we have proved the following convergence to zero for any $\delta$:
		\begin{equation}
			\lim_{\eta \to 0} [\mu \otimes \nu^N](\cG[\delta, \eta]) = 0.
			\label{eq_munuN_cGde}
		\end{equation}
		
		\subsubsection*{Step 3: Relation between the level sets of $\eps^N$ to $\cG[\delta, \eta]$}\hfill\\	
		We consider for any value $\delta>0$ the corresponding level-set of $\eps^N$:
		\begin{equation}
			\cH^N[\delta]
			:=	\{x\in \bX;\, \eps^N(x)\ge \delta\}\,.
			\label{eq_def_Hde}
		\end{equation}
	Let $\theta := 2 + 2\nu(\cY)>0$.
		For any $\eta$ sufficiently small, 
		we will relate in the following lemma the intersection $\cH^N[\delta]^c\cap \cH^N[\theta\,\delta]^\eta$
		to conditions on the following subsets of $\cY$:
		\begin{equation}
			\cG_x[\delta, \eta]
			:=\{y\in \cY;\, 
			(x, y) \in \cG[\delta, \eta]\},
			\label{eq_def_cGdx}
		\end{equation}
		namely the restriction of $\cG[\delta, \eta]$,
		recall \eqref{eq_def_cGde},
		with $x$ as the first coordinate.
		
		\begin{lem}
			\label{lem_incl_Hn}
			The following inclusion holds for any $\delta, \eta>0$
			and $N\ge 1$:
			\begin{equation*}
				\cH^N[\delta]^c\cap \cH^N[\theta\,\delta]^\eta
				\subset \Big\{x\in \bX;\, 
				\nu^N(\cG_x[\delta, \eta])\ge \frac{\delta}{2}
				\Big\}\cup \Big\{x\in \bX;\, 
				\nu(\cG_x[\delta, \eta])\ge \frac{\delta}{2}
				\Big\}.
			\end{equation*}
		\end{lem}
		
		\begin{proof}
			Let us consider any $x\in \cH^N[\delta]^c\cap \cH^N[\theta\,\delta]^\eta$.
			We can thus choose some $x'\in \cH^N[\theta\,\delta] \cap B(x, \eta)$.
			By virtue of \eqref{eq_def_cGdx}, for $N$ sufficiently large, we obtain 
\begin{equation}
\begin{split}
		\Big|\int_{\cG_x[\delta, \eta]^c}
	|k(x, y) - k(x', y)|\, [\nu^N - \nu](dy)
	\Big|
	&\le [\nu^N \vee \nu](\cG_x[\delta, \eta]^c)\cdot \delta 
	\\&\le 2\delta \cdot \nu(\cY)\,,
\end{split}
	\label{eq_prop_cGc}
\end{equation}
			where we have exploited that $k$, $\nu^N$ and $\nu$ are non-negative
		and that $\nu^N(\cY)$ converges to $\nu(\cY)>0$.
			On the other hand, 
			\begin{equation}
				\Big|\int_{\cG_x[\delta, \eta]}
				|k(x, y) - k(x', y)|\, [\nu^N - \nu](dy)
				\Big|
				\le [\nu^N \vee \nu](\cG_x[\delta, \eta]),
				\label{eq_prop_cGp}
			\end{equation}
			where we recall the assumption that $k$ is non-negative and bounded by 1.
			Since $x\in \cH^N[\delta]^c$ while $x'\in \cH^N[\theta\,\delta]$:
			\begin{equation*}
				|\eps^N(x) - \eps^N(x')|\ge \eps^N(x') - \eps^N(x)
				\ge  \delta\cdot(1+2\nu(\cY)).
			\end{equation*}
			Recalling \eqref{eq_def_epsN} to combine this result 
			with \eqref{eq_prop_cGc} and \eqref{eq_prop_cGp},
			we deduce:
			\begin{equation*}
				[\nu \vee \nu^N](\cG_x[\delta, \eta]) \ge \delta.
				\label{eq_prop_cGx}
			\end{equation*}
This inequality implies either $\nu^N(\cG_x[\delta, \eta])\ge \delta/2$
or  $\nu(\cG_x[\delta, \eta])\ge \delta/2$.
This concludes the proof of Lemma~\ref{lem_incl_Hn}.
		\end{proof}
		
		\subsubsection*{Step 4: Proof of Proposition~\ref{prop_cvg_mu0}
		in the particular case where $(\mu^N)$, $(\nu^N)$
		are deterministic.}
		\hfill\\
		For any $\delta>0$ and $N\ge 1$ sufficiently large, since the mass of $\mu^N$
		converges to the one of $\mu$, we obtain 
		\begin{equation}
			\langle \mu^N, \eps^N\rangle 
			\le 2\theta\,\delta\cdot \mu(\bX)
			+ \mu^N(\cH^N[\theta\,\delta]),
			\label{eq_prop_muNeN}
		\end{equation}
		where we recall \eqref{eq_def_Hde}.
		
		We choose $\eta\in (0, \delta)$ sufficiently small
		thanks to Step 2,
		to ensure both that $[\mu \otimes \nu](\cG[\delta, \eta])$ 
		is smaller than $\delta^2/2$
		and similarly for $[\mu \otimes \nu^N](\cG[\delta, \eta])$
		for any $N$ sufficiently large.
		Thanks to the Markov inequality, we have 
		\begin{equation}
			\mu\Big(\Big\{x\in \bX;\, 
			\nu(\cG_x[\delta, \eta])\ge \frac{\delta}{2}
			\Big\}\Big)
			\le \frac{2 [\mu \otimes \nu](\cG[\delta, \eta])}{\delta}
			\le \delta.
			\label{eq_prop_nucGx}
		\end{equation}
		Similarly, 
		\begin{equation}
			\mu\Big(\Big\{x\in \bX;\, 
			\nu^N(\cG_x[\delta, \eta])\ge \frac{\delta}{2}
			\Big\}\Big)
			\le \frac{2 [\mu \otimes \nu^N](\cG[\delta, \eta])}{\delta}
			\le \delta.
			\label{eq_prop_nuNcGx}
		\end{equation}
		Since $\mu^N$ converges weakly to $\mu$,
		for any $N$ sufficiently large  we have 
		\begin{equation}
			\mu^N(\cH^N[\theta\,\delta])
			\le \mu(\cH^N[\theta\,\delta]^\eta) + \eta,
			\label{eq_prop_eta}
		\end{equation}
		where we adapt the definition of $\eta$ vicinity given in \eqref{eq_def_Aeta}
		to subsets of $\bX$.
		As we expect $\cH^N[\theta\,\delta]^\eta$ to be mostly comprised into $\cH^N[\delta]$,
		we make the following distinction
		\begin{equation}
			\mu(\cH^N[\theta\,\delta]^\eta)
			\le \mu(\cH^N[\delta]) + \mu(\cH^N[\delta]^c\cap \cH^N[\theta\,\delta]^\eta).
			\label{eq_prop_Hde}
		\end{equation}
		Thanks to the Markov inequality, we obtain 
		\begin{equation*}
			\mu(\cH^N[\delta]) \le \delta^{-1}\cdot \langle \mu, \eps^N\rangle,
		\end{equation*}
		which converges to 0 as $N$ tends to infinity as stated in \eqref{eq_cvg0_muEpsN}.
		We thus restrict to $N$ sufficiently large in order to ensure that 
		\begin{equation}
			\mu(\cH^N[\delta]) \le \delta.
			\label{eq_prop_muHne2}
		\end{equation}
		On the other hand, as a consequence of Step 3, see Lemma~\ref{lem_incl_Hn}, we obtain 
		\begin{multline}
			\mu(\cH^N[\delta]^c\cap \cH^N[\theta\,\delta]^\eta)
			\\	\le \mu\Big(\Big\{x\in \bX;\, 
			\nu^N(\cG_x[\delta, \eta])\ge \frac{\delta}{2}
			\Big\}\Big)
			+ \mu\Big(\Big\{x\in \bX;\, 
			\nu(\cG_x[\delta, \eta])\ge \frac{\delta}{2}
			\Big\}\Big).
			\label{eq_prop_hNd}
		\end{multline}
		For the next upper-bound, 
		valid for $\eta$ sufficiently small then $N$ sufficiently large,
		we recall 
		\eqref{eq_prop_nucGx},
		\eqref{eq_prop_nuNcGx}, 
		\eqref{eq_prop_Hde},\eqref{eq_prop_muHne2},
		\eqref{eq_prop_hNd} and get 
		\begin{equation*}
			\mu(\cH^N[\theta\,\delta]^\eta)
			\le 3 \delta.
		\end{equation*}
		It remains to combine this result with \eqref{eq_prop_muNeN} and \eqref{eq_prop_eta} to conclude the proof that $(\langle \mu^N, \eps^N\rangle)$
		tends to zero as $N$ tends to infinity,
		since $\delta$ can be taken arbitrarily small.
		This concludes the proof of Proposition~\ref{prop_cvg_mu0}
		in the particular case where the sequences $(\mu^N)$ and $(\nu^N)$ are deterministic.

	\subsubsection*{Step 5: Proof of Proposition~\ref{prop_cvg_mu0}
	in the general case where $(\mu^N)$ and $(\nu^N)$ are random.}
\hfill\\
For this final step, we no longer require the sequences $(\mu^N)$ and $(\nu^N)$
to be a priori deterministic, 
though our approach consists in referring to this convenient situation. 
		We exploit \cite[Lemma 4.2]{Ka02}
		to relate the convergence of probability to a.s. convergence of sequence extractions.
		Let $(N[k])_{k\ge 1}\in \mathbb{N}^\mathbb{N}$ be an increasing sequence.
		Since $\mu^{N}$ and $\nu^N$
		converge in probability, 
		we can extract a subsequence $(\check N[\ell])_{\ell\ge 1} = (N[K[\ell]])_{\ell\ge 1}$
		from this sequence 
		such that $(\check\mu^{\ell}) = (\mu^{\check N[\ell]})$ 
		and $(\check\nu^{\ell}) = (\nu^{\check N[\ell]})$ 
		converge a.s. respectively to  $\mu$ and $\nu$.
		On this event of probability 1,
		we deduce from Step 4
		that $\langle \check\mu^{\ell}, \check \eps^\ell\rangle$
		converges to zero as $\ell$ tends to infinity.
		Since the convergence of such an extraction
		of the sequence $(\langle \mu^N, \eps^N\rangle)$
		holds whatever the initial extraction,
		this concludes the proof of Proposition  \ref{prop_cvg_mu0}
		in that $(\langle \mu^N, \eps^N\rangle)$ converges in probability to 0.
	\end{proof}
	\medskip
	{\bf Conflict of interest} The authors declare that they have no conflict of interest.
	
	{\bf Funding declaration} This research was not supported by any specific funding.

	\section*{Acknowledgement} 
	We thank the reviewer for the helpful comments that have led to substantial
improvements of the exposition of the paper.

	\bibliographystyle{abbrv}
	\bibliography{GVI_S_filter.bib}

\end{document}